\documentclass[11pt]{amsart}

\usepackage{amsmath, amsthm, amssymb, color, graphicx}
\usepackage[margin=1.4in]{geometry}
\usepackage{tikz}
\tikzstyle{bv}=[circle,draw=black!90,fill=black!100,thick,inner sep=2pt,minimum width=5pt]
\usepackage{hyperref}
\usepackage[T1]{fontenc}

\usepackage[latin1]{inputenc}
\usepackage{paralist}
\usepackage{subfigure}

\newtheorem{thm}{Theorem}
\newtheorem{theorem}[thm]{Theorem}

\newtheorem{prop}[thm]{Proposition}
\newtheorem{lemma}[thm]{Lemma}
\newtheorem{cor}[thm]{Corollary}
\newtheorem{corollary}[thm]{Corollary}

\newtheorem{claim}[thm]{Claim}

\newcommand{\cc}{{\boldsymbol{c}}}

\newcommand{\mm}{{\boldsymbol{m}}}
\newcommand{\nn}{{\boldsymbol{n}}}

\newcommand{\xx}{{\boldsymbol{x}}}
\newcommand{\zz}{{\boldsymbol{z}}}

\newcommand{\wf}{\widetilde{f}}

\newcommand{\bN}{\textbf{N}}
\newcommand{\NN}{\textbf{N}}
\newcommand{\CC}{\textbf{C}}

\newcommand{\mC}{\mathcal{C}}

\newcommand{\mT}{\mathcal{T}}
\newcommand{\mTT}{\mathcal{T}}
\newcommand{\mP}{\mathcal{P}}
\newcommand{\mL}{\mathcal{L}}
\newcommand{\mF}{\mathcal{F}}
\newcommand{\mU}{\mathcal{U}}

\DeclareMathOperator{\Cay}{\textrm{Cayley}}
\newcommand{\Cayd}{\Cay(d,\rho)}
\newcommand{\Caydp}{\Cay(d+1,d+1)}

\newcommand{\lex}{\textrm{lex}}

\newcommand{\GG}{g}
\DeclareMathOperator{\ch}{in}
\DeclareMathOperator{\inj}{inj}

\DeclareMathOperator{\roo}{root}

\newcommand{\ga}{\gamma}
\newcommand{\de}{\delta}

\newcommand{\ka}{\kappa}
\newcommand{\om}{\omega}

\newcommand{\Ga}{\Gamma}

\newcommand{\vect}[1]{\overrightarrow #1}

\newcommand{\ds}{\displaystyle}
\newcommand{\fig}[3]{\begin{figure}[h]\begin{center}\includegraphics[#1]{#2}\end{center}\caption{#3}\label{fig:#2}\end{figure}}

\title{Counting trees using symmetries}

\author{Olivier Bernardi and Alejandro H. Morales}
\thanks{O.B. aknowledges support from NSF grant DMS-1068626, ANR A3, and ERC Explore-Maps}
\date{\today}

\begin{document}
\setcounter{tocdepth}{2}

\begin{abstract}
We prove a new formula for the generating function of multitype Cayley trees counted according to their degree distribution. Using this formula we recover and extend several enumerative results about trees. In particular, we extend some results by Knuth and by Bousquet-Mélou and Chapuy about embedded trees. We also give a new proof of  the multivariate Lagrange inversion formula. 
 Our strategy for counting trees is to exploit symmetries of refined enumeration formulas: proving these symmetries is easy, and once the symmetries are proved the formulas follow effortlessly. We also adapt this strategy to recover an enumeration formula of Goulden and Jackson for cacti counted according to their degree distribution.
\end{abstract}

\maketitle

\section{Introduction}
The enumeration of trees is a very classical subject. For instance, there is a well-known formula for the number of \emph{unitype Cayley trees}. Recall that a unitype Cayley tree with $n$ vertices is a connected acyclic graph with vertex set $[n]=\{1,\ldots,n\}$. There are $n^{n-2}$ such trees, and there is a very simple formula for the generating function of Cayley trees counted according to their degree distribution. Namely,
\begin{eqnarray}\label{eq:Cayley-unitype}
\sum_{T\textrm{ Cayley tree}\atop \textrm{with vertex set }[n]}x_1^{\deg(1)}x_2^{\deg(2)}\cdots x_n^{\deg(n)}=x_1x_2\cdots x_n(x_1+x_2+\cdots+ x_n)^{n-2},
\end{eqnarray}
where $\deg(i)$ is the degree of vertex $i$. 

In this paper we consider \emph{multitype Cayley trees}, that is, trees in which vertices have both a \emph{type} and a \emph{label}. We obtain a formula extending~\eqref{eq:Cayley-unitype} from the unitype setting to the multitype setting (Theorem~\ref{thm:multitype-Cayley}). More precisely, our formula gives the generating function of rooted multitype Cayley trees counted according to the number of children of each type of each vertex.  
Our formula is surprisingly simple, and from it we derive many enumerative corollaries in Section~\ref{sec:applications}. In particular, we recover and extend the results of Knuth~\cite{Knuth:enumeration-trees}, and the recent results of Bousquet-Mélou and Chapuy~\cite{MBM-Chapuy:vertical-profile-trees} about ``embedded trees''. We also obtain a short proof of the multivariate Lagrange inversion formula~\cite{Gessel:multivariate-Lagrange} in Section~\ref{sec:Lagrange}.   Our strategy for counting trees is to exploit symmetries of refined enumeration formulas, and we also use this strategy in order to recover a formula of Goulden and Jackson for counting cacti according to their degree distribution in Section~\ref{sec:cacti}. 
We mention lastly that because we count trees according to their vertex degrees, our results could equivalently be stated in terms of plane trees instead of Cayley trees (see Section~\ref{sec:cacti} for a more detailed discussion). Also, our results can easily be extended in order to count rooted forests (see Corollary~\ref{cor:multitype-forest}).

In order to illustrate our approach for counting trees, we give a new proof of~\eqref{eq:Cayley-unitype}. There are already many beautiful proofs of this formula including Pr\"ufer's code bijection~\cite{Prufer:Cayley}, Joyal's endofunction approach~\cite{Joyal:Cayley}, Pitman's double counting argument~\cite{JP}, the matrix-tree theorem~\cite[Chapter 5]{JWM}, and recursive approaches~\cite[Chapter 5.3]{EC2}. Our method is different: we start by proving the ``symmetries'' in formula~\eqref{eq:Cayley-unitype} and use them at our advantage in order to enumerate Cayley trees.

First observe that a Cayley tree with $n$ vertices has $n-1$ edges, hence the degrees of its vertices sum to $2n-2$. Given a tuple of positive integers $\ga=(\ga_1,\ldots,\ga_n)$ summing to $2n-2$, we denote by $\mTT_\ga$ the set of Cayley trees with $n$ vertices such that vertex~$i$ has degree $\ga_i$ for all $i\in[n]$. We first claim that the cardinalities of the sets $\mTT_\ga$ are related to one another by simple factors:
\begin{lemma}\label{lem:symmetry-Cayley-trees}
Let $\ga=(\ga_1,\ldots,\ga_n)$ be tuple of positive integers summing to $2n-2$. Let $i, j\in[n]$ and let $\ga'=(\ga_1',\ldots,\ga_n')$ be defined by $\ga_i'=\ga_i-1$, $\ga_j'=\ga_j+1$ and $\ga_k'=\ga_k$ for $k\neq i,j$.
Then 
$$(\ga_i-1)|\mTT_\ga|=(\ga_j'-1)|\mTT_{\ga'}|.$$
\end{lemma}
\begin{proof}
The proof is illustrated in Figure~\ref{fig:Symmetry-Cayley}. Let $\mTT^{i,j}_\ga$ be the set of trees in $\mTT_\ga$ with a marked edge incident to vertex~$i$ not in the path between vertices~$i$ and~$j$. Clearly $|\mTT^{i,j}_{\ga}|=(\ga_i-1)|\mTT_\ga|$. Moreover, there is an obvious bijection $\Phi$ between $\mTT^{i,j}_\ga$ and $\mTT^{j,i}_{\ga'}$: given a marked tree $T\in\mTT^{i,j}_{\ga}$, the tree $\Phi(T)\in\mTT^{j,i}_{\ga'}$ is obtained by ungluing the marked edge from vertex~$i$, and gluing it to vertex~$j$.
\end{proof}

\fig{width=11cm}{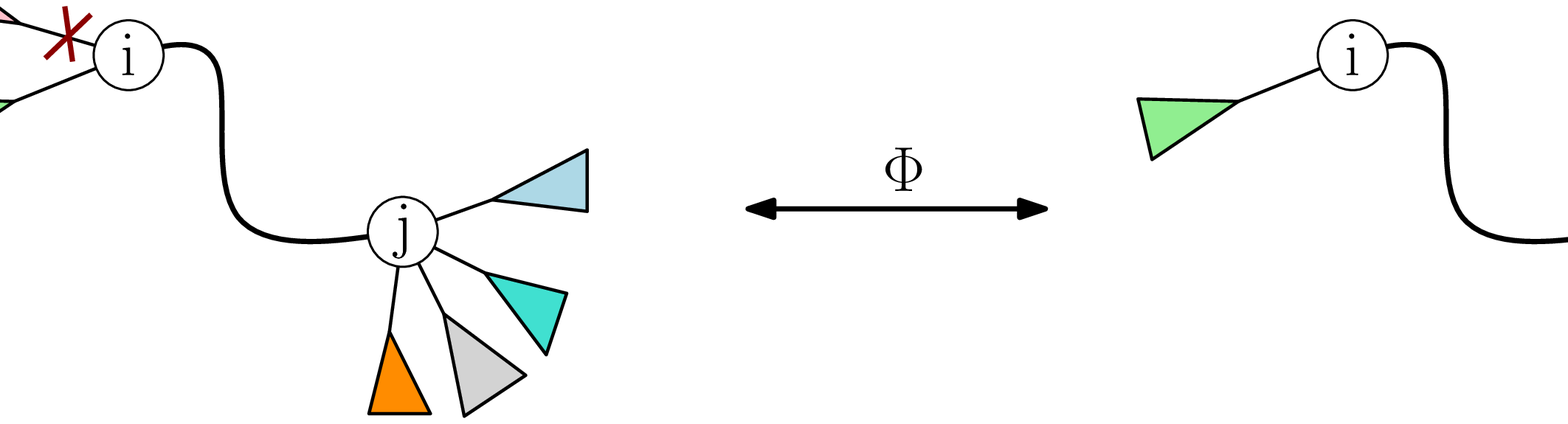}{The bijection $\Phi$ between the sets $\mTT^{i,j}_\ga$ and $\mTT^{j,i}_{\ga'}$.}
Using Lemma~\ref{lem:symmetry-Cayley-trees} repeatedly, we can express $|\mTT_\ga|$ in terms of $|\mTT_\ka|$, where $\ka=(n-1,1,1,\ldots,1)$. Indeed,
\begin{equation}\label{eq:repeat-to-star}
\begin{split}
|\mTT_\ga|&=\frac{\ga_1(\ga_1+1)\cdots (\ga_1+\ga_2-2)}{(\ga_2-1)!}|\,\mTT_{\ga_1+\ga_2-1,1,\ga_3,\ldots,\ga_n}|\\
&=\frac{\ga_1(\ga_1+1)\cdots (\ga_1+\ga_2+\cdots+\ga_n-n)}{(\ga_2-1)!\,(\ga_3-1)!\,\cdots\,(\ga_n-1)!}\,|\mTT_{\ka}|\\
&={n-2 \choose \ga_1-1,\ga_2-1,\ldots,\ga_n-1}\,|\mTT_{\ka}|.
\end{split}
\end{equation}
Moreover, $|\mTT_\ka|=1$ (only one ``star tree''), hence $\displaystyle |\mTT_\ga|={n-2 \choose \ga_1-1,\ga_2-1,\ldots,\ga_n-1}$.
This implies~\eqref{eq:Cayley-unitype} since this multinomial is the coefficient of $x_1^{\ga_1}x_2^{\ga_2}\cdots x_n^{\ga_n}$ in $x_1x_2\cdots x_n(x_1+x_2+\cdots+x_n)^{n-2}$.

We will now use the above philosophy for tackling more advanced counting problems of tree-like structures.

\section{The generating function of multitype Cayley trees}\label{sec:multitype-Cayley}
Let $\nn=(n_1,n_2,\ldots,n_d)$ be a tuple of non-negative integers. 
A \emph{multitype Cayley tree} of \emph{profile} $\nn$ is a tree (i.e., acyclic connected graph) with vertex set 
$$V_{\nn}=\{(t,i),~t\in[d],i\in[n_t]\}.$$
The vertex $(t,i)\in V_\nn$ is said to have \emph{type} $t$ and \emph{label} $i$. A multitype Cayley tree is represented in Figure~\ref{fig:example-multitype-Cayley}. 

\fig{width=9cm}{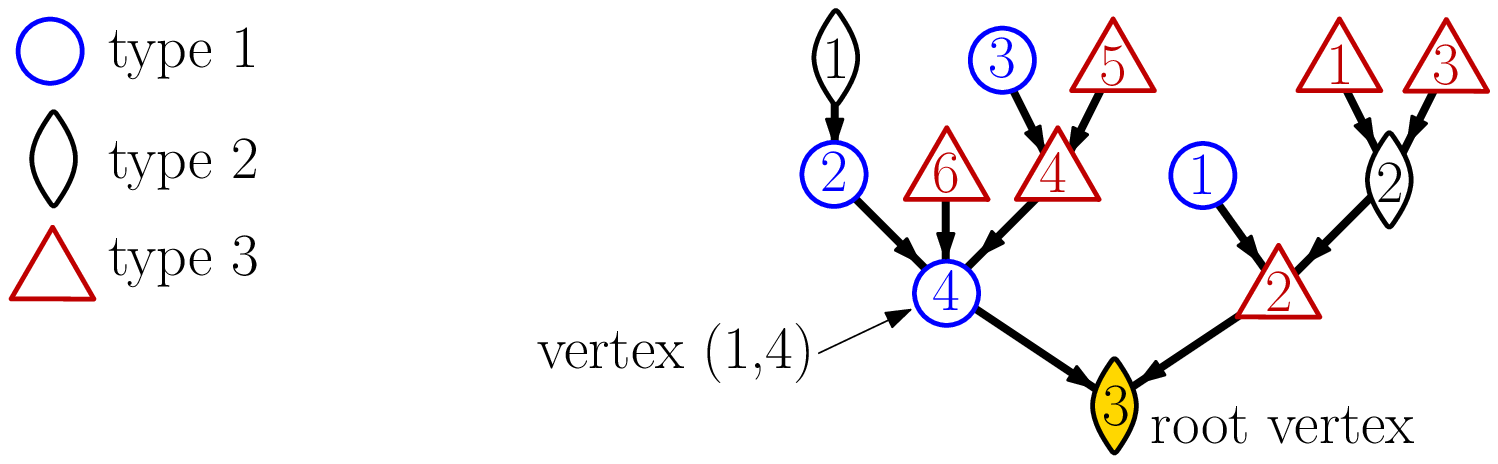}{A multitype Cayley tree of profile $\nn=(4,3,6)$ rooted on a vertex of type $2$. The shape of the vertices indicate their type, while the numbers indicate their label. For the vertex $(1,4)$ the indegrees are $\ch_1(1,4)=1$, $\ch_2(1,4)=0$, and $\ch_3(1,4)=2$.} 

A multitype Cayley tree is said to be \emph{rooted} if one of the vertices is distinguished as the \emph{root vertex}; in this case the edges of the tree are oriented toward the root vertex. 
We denote by $\mT_\rho(\nn)$ the set of rooted multitype Cayley trees of profile $\nn$ in which the root vertex has type $\rho$. 
Given a tree $T\in \mT_\rho(\nn)$, denote by $\ch_s(t,i)$ the number of children of type $s$ of the vertex $(t,i)$. The tuple of integers $(\ch_s(t,i))_{s,t\in [d],i\in[n_t]}$ is called the \emph{indegree vector} of $T$. We now state our main result. 

\begin{thm}\label{thm:multitype-Cayley}
Let $\nn=(n_1,n_2,\ldots,n_d)$ be a tuple of positive integers, and let $\rho\in [d]$. The generating function of rooted multitype Cayley trees of profile $\nn$ with root vertex of type $\rho$ counted according to their indegree vectors is
\begin{eqnarray}\label{eq:GF-multitype-Cayley}
\sum_{T\in\mT_\rho(\nn)}~\prod_{s,t\in[d],i\in[n_t]}x_{s,t,i}^{\ch_s(t,i)}&=&\prod_{s\in[d]}\Bigg(\sum_{t\in[d],i\in[n_t]}x_{s,t,i}\Bigg)^{n_s-1}\times \Delta,
\end{eqnarray}
where 
$$\Delta=\sum_{A\in\Cayd} ~\prod_{(s,t)\in A}\Bigg(\sum_{i\in [n_t]}x_{s,t,i}\Bigg).$$
where $\Cayd$ is the set of unitype Cayley trees $A$ with vertex set $[d]$ rooted at vertex~$\rho$ and considered as oriented toward its root vertex, and the notation $(s,t)\in A$ means that the oriented edge $(s,t)$ belongs to the oriented tree $A$.
\end{thm}

\noindent \textbf{Remark.} The set $\Cayd$ has cardinality $d^{d-2}$. This is the set of spanning trees of the complete graph $K_d$, with root vertex $\rho$, hence one can express the sum $\Delta$ appearing in~\eqref{eq:GF-multitype-Cayley} as a determinant by using the matrix-tree theorem~\cite[Theorem 5.6.8]{EC2}. More precisely, let $L$ be the $d\times d$ matrix with entries $\ell_{s,t}=-\sum_{i\in[n_t]}x_{s,t,i}$ if $s\neq t$, and $\ell_{s,s}=-\sum_{t\in[d],t\neq s}\ell_{s,t}$ for all $s\in[d]$. Then $\Delta$ is the determinant of the matrix obtained by deleting the $\rho$th row and $\rho$th column of $L$.\\

The case $d=1$ of Theorem~\ref{thm:multitype-Cayley} corresponds to the enumeration of unitype \emph{rooted} Cayley trees (i.e., rooted spanning trees of the complete graph $K_n$) according to the indegree of vertices. Indeed, upon setting $d=1$, $n_1=n$ and $x_{1,1,i}=x_i$ for all $i\in[n]$ in~\eqref{eq:GF-multitype-Cayley} one gets the well-known formula
\begin{equation*}
\sum_{T\textrm{ rooted Cayley tree}\atop \textrm{with vertex set }[n]}x_1^{\ch(1)}x_2^{\ch(2)}\cdots x_n^{\ch(n)}=(x_1+x_2+\cdots+ x_n)^{n-1},
\end{equation*}
which is easily seen to be equivalent to~\eqref{eq:Cayley-unitype}.
In the case $d=2$, Theorem~\ref{thm:multitype-Cayley} can be specialized to give the generating function of the spanning trees of the complete bipartite graph $K_{m,n}$ counted according to the indegree of vertices; see e.g.~\cite[Exercise 5.30]{EC2}. Indeed, upon setting $d=2$, $n_1=m$, $n_2=n$, and $x_{2,1,i}=x_i$, $x_{1,2,j}=y_j$, $x_{1,1,i}=x_{2,2,j}=0$ for all $i\in[m],j\in[n]$ one gets:
\[
\sum_{T\subset K_{m,n}}~\prod_{i\in[m]}x_{i}^{\ch(1,i)}\prod_{j\in[n]}y_{j}^{\ch(2,j)}=(x_1+x_2+\cdots+x_m)^n(y_1+y_2+\cdots+y_n)^{m-1},
\]
where the sum is over all the spanning trees of $K_{m,n}$ rooted on a vertex of type $1$. Many more applications are discussed in Section~\ref{sec:applications}.\\

Before proving Theorem \ref{thm:multitype-Cayley}, we mention a corollary about rooted forests. Recall that a \emph{forest} is an acyclic graph (hence each component is a tree), and that a forest is said to be \emph{rooted} if each connected component has a vertex distinguished as the \emph{root vertex}. We denote by $\mF(\nn)$ the set of \emph{rooted multitype forests of profile} $\nn=(n_1,\ldots,n_d)$, that is, the set of rooted forests with vertex set $\{(t,i),~t\in[d],i\in[n_t]\}$. For $F\in \mF(\nn)$ we think of each connected component of $F$ as being oriented toward its root vertex, and we denote by $\ch_s(t,i)$ the number of children of type $s$ of the vertex $(t,i)$.
\begin{cor}\label{cor:multitype-forest}
Let $\nn=(n_1,\ldots,n_d)$ be a tuple of positive integers, and let 
$$\mF_\nn(\xx,\zz)=\sum_{F\in\mF(\nn)}~\prod_{s\in[d]}z_s^{\roo_s(F)}\prod_{s,t\in[d],i\in[n_t]}x_{s,t,i}^{\ch_s(t,i)},$$
where $\roo_s(F)$ is the number of root vertices of type $s$ in the rooted forest $F$. Then,
\begin{eqnarray}\label{eq:GF-multitype-forest}
\mF_\nn(\xx,\zz)=\prod_{s\in[d]}\Bigg(z_s+\sum_{t\in[d],i\in[n_t]}x_{s,t,i}\Bigg)^{n_s-1}\times \Gamma,
\end{eqnarray}
where 
$$\Gamma=\sum_{B\in F_d} ~\Bigg(\prod_{s\textrm{ root vertex of }B }z_s \Bigg)\prod_{(s,t)\in B}\Bigg(\sum_{i\in [n_t]}x_{s,t,i}\Bigg),$$
where $F_d$ is the set of rooted forests with vertex set $[d]$ (where each connected component is thought as oriented toward its root vertex) and the notation $(s,t)\in B$ means that the oriented edge $(s,t)$ belongs to the oriented forest $B$.
\end{cor}
\noindent \textbf{Remark.} The set $\roo_s(F)$ has cardinality $(d+1)^{d-1}$. Using the (forest version of the) matrix-tree theorem one can express the sum $\Gamma$ appearing in~\eqref{eq:GF-multitype-forest} as a determinant. More precisely $\Gamma$ is the determinant of the $d\times d$ matrix $L$ with entries $\ell_{s,t}=-\sum_{i\in[n_t]}x_{s,t,i}$ if $s\neq t$, and $\ell_{s,s}=z_s-\sum_{t\in[d],t\neq s}\ell_{s,t}$ for all $s\in[d]$.\\

\begin{proof}[Proof of Corollary \ref{cor:multitype-forest}]
There is bijection between the set $\mF(\nn)$ of rooted forests and the set  $\mT_{d+1}(\nn')$ of rooted trees where $\nn'=(n_1,\ldots,n_{d},1)$. Indeed, given a forest $F\in\mF(\nn)$ one gets a rooted tree $T\in\mT_{d+1}(\nn')$ by joining all the root vertices of $F$ to a new vertex of type $d+1$ which becomes the root vertex of $T$. Thus,
$$\mF_\nn(\xx,\zz)=\sum_{T\in\mT_{d+1}(\nn')}~\prod_{s\in[d]}z_s^{\ch_s(d+1,1)}\prod_{s,t\in[d],i\in[n_t]}x_{s,t,i}^{\ch_s(t,i)}.$$
Hence we can obtain the desired expression for $\mF_\nn(\xx,\zz)$ by using \eqref{eq:GF-multitype-Cayley} for the profile $\nn'=(n_1,\ldots,n_{d},1)$ and setting $x_{s,d+1,1}=z_s$. This immediately gives \eqref{eq:GF-multitype-forest} because, via the bijection between $F_d$ and $\Caydp$, one has
$$\Gamma= \sum_{A\in\Caydp} ~\Bigg(\prod_{(s,d+1)\in A}z_s\Bigg)\prod_{(s,t)\in A,t\in [d]}\Bigg(\sum_{i\in [n_t]}x_{s,t,i}\Bigg).$$
\end{proof}

The rest of this section is devoted to the proof of Theorem~\ref{thm:multitype-Cayley}. We start with the analogue of Lemma~\ref{lem:symmetry-Cayley-trees}.
Given a tuple $\ga=(\ga_{s,t,i})_{s,t\in [d],i\in[n_t]}$ of non-negative integers, we denote by $\mT_{\rho,\ga}$ the set of trees is $T\in\mT_\rho(\nn)$ having indegree vector $\ga$, that is, satisfying $\ch_s(t,i)=\ga_{s,t,i}$ for all $s,t\in[d],i\in[n_t]$.
Observe that this set is empty unless,
$$ n_s=\delta_{s,\rho}+\sum_{t\in[d]}\sum_{i\in[n_t]}\ga_{s,t,i},$$
for all $s\in [d]$, where $\delta_{s,\rho}$ denotes the Kronecker delta.

\begin{lemma}\label{lem:symmetry-Cayley-multitype}
Let $\nn=(n_1,n_2,\ldots,n_d)$ be a tuple of positive integers and let $\ga=(\ga_{s,t,i})_{s,t\in [d],i\in[n_t]}$ be a tuple of non-negative integers. Let $s,t\in[d], i\neq j\in[n_t]$, and let $\ga'=(\ga_{s,t,i}')_{s,t\in [d],i\in[n_t]}$ be defined by $\ga_{s,t,i}'=\ga_{s,t,i}-1$, $\ga_{s,t,j}'=\ga_{s,t,j}+1$ and $\ga_{a,b,c}'=\ga_{a,b,c}$ for $(a,b,c)\notin \{(s,t,i),(s,t,j)\}$. Then 
$$\ga_{s,t,i}|\mT_{\rho,\ga}|=\ga_{s,t,j}'|\mT_{\rho,\ga'}|.$$
\end{lemma}

\begin{proof}
Let $\mT^{s,t,i}_{\rho,\ga}$ be the set of trees in $\mT_{\rho,\ga}$ with a marked edge joining the vertex $(t,i)$ to one of its children of type $s$. Clearly, $|\mT^{s,t,i}_{\rho,\ga}|=\ga_{s,t,i}|\mT_{\rho,\ga}|$, hence it suffices to exhibit a bijection between $\mT^{s,t,i}_{\rho,\ga}$ and $\mT^{s,t,j}_{\rho,\ga'}$. 
We first partition $\mT^{s,t,i}_{\rho,\ga}$ into two sets $\widehat{\mT}^{s,t,i}_{\rho,\ga}$ and $\widetilde{\mT}^{s,t,i}_{\rho,\ga}$ defined as follows: a tree $T\in\mT^{s,t,i}_{\rho,\ga}$ is in $\widehat{\mT}^{s,t,i}_{\rho,\ga}$ if the marked edge $e$ of $T$ is not on the path of $T$ between $(t,i)$ and $(t,j)$ (equivalently, $e$ is not on the path from $(t,j)$ to the root vertex), and the tree $T$ is in $\widetilde{\mT}^{s,t,i}_{\rho,\ga}$ otherwise. We now describe a bijection $\widehat{\Phi}_{s,t,i,j}$ between $\widehat{\mT}^{s,t,i}_{\rho,\ga}$ and $\widehat{\mT}^{s,t,j}_{\rho,\ga'}$, and a bijection $\widetilde{\Phi}_{s,t,i,j}$ between $\widetilde{\mT}^{s,t,i}_{\rho,\ga}$ and $\widetilde{\mT}^{s,t,j}_{\rho,\ga'}$. These bijections are represented in Figure~\ref{fig:Symmetry-Cayley-multitype}. 

Given a tree $T$ in $\widehat{\mT}^{s,t,i}_{\rho,\ga}$, the tree $\widehat{\Phi}_{s,t,i,j}(T)$ is obtained by ungluing the marked edge from vertex $(t,i)$ and gluing it to vertex $(t,j)$. It is clear that $\widehat{\Phi}_{s,t,i,j}=\widehat{\Phi}_{s,t,j,i}^{-1}$, hence $\widehat{\Phi}_{s,t,i,j}$ is a bijection between the sets $\widehat{\mT}^{s,t,i}_{\rho,\ga}$ and $\widehat{\mT}^{s,t,j}_{\rho,\ga'}$. 

Given a tree $T$ in $\widetilde{\mT}^{s,t,i}_{\rho,\ga}$, the tree $\widetilde{\Phi}_{s,t,i,j}(T)$ is obtained by ungluing all the unmarked edges oriented toward the vertex $(t,i)$ and gluing them to $(t,j)$, ungluing all the edges originally oriented toward the vertex $(t,j)$ and gluing them to $(t,i)$, and finally relabeling the vertex $(t,i)$ as $(t,j)$ and vice-versa; see Figure~\ref{fig:Symmetry-Cayley-multitype}. It is clear that $\widetilde{\Phi}_{s,t,i,j}=\widetilde{\Phi}_{s,t,j,i}^{-1}$, hence $\widetilde{\Phi}_{s,t,i,j}$ is a bijection between the sets $\widetilde{\mT}^{s,t,i}_{\rho,\ga}$ and $\widetilde{\mT}^{s,t,j}_{\rho,\ga'}$. 
\end{proof}
\fig{width=\linewidth}{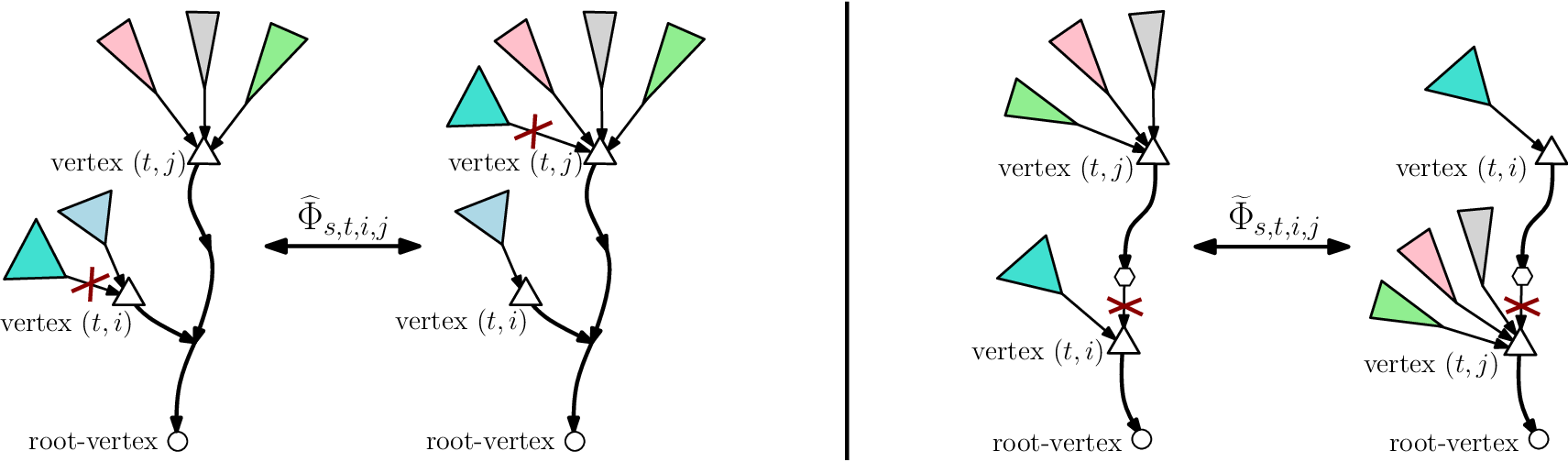}{The bijections $\widehat{\Phi}_{s,t,i,j}$ (left) and $\widetilde{\Phi}_{s,t,i,j}$ (right).}

A multitype Cayley tree is called a \emph{star tree} if all the vertices not labeled 1 are leaves, that is, if $\ch_s(t,i)=0$ for all $s,t\in [d]$ and all $i\neq 1$. 
The following Lemma shows that the problem of enumerating multitype Cayley trees reduces to the problem of enumerating star trees.

\begin{lemma}\label{lem:tostars}
Let $\nn=(n_1,n_2,\ldots,n_d)$ be a tuple of positive integers. Let $\ga=(\ga_{s,t,i})_{s,t\in [d],i\in[n_t]}$ be a tuple of non-negative integers and let $\ga^*=(\ga^*_{s,t,i})_{s,t\in [d],i\in[n_t]}$ be defined for all $s,t\in [d]$ by $\ga^*_{s,t,1}=\sum_{i\in[n_t]}\ga_{s,t,i}$ and $\ga^*_{s,t,i}=0$ for all $i\neq 1$. The number of trees of indegree vector $\ga$ is 
\begin{eqnarray}\label{eq:reduce-to-star}
|\mT_{\rho,\ga}|&=&|\mT_{\rho,\ga^*}|\times \prod_{s,t\in[d]}{\ga^*_{s,t,1} \choose \ga_{s,t,1},\ldots,\ga_{s,t,n_t}}.
\end{eqnarray}
Equivalently, in terms of generating functions,
\begin{eqnarray}\label{eq:reduce-to-starGF}
\sum_{T\in\mT_\rho(\nn)~}\prod_{s,t\in[d],i\in[n_t]}x_{s,t,i}^{\ch_s(t,i)}&=& \sum_{T\emph{ star tree in } \mT_\rho(\nn)~}\prod_{s,t\in[d]}\Bigg(\sum_{i\in[n_t]}x_{s,t,i}\Bigg)^{\ch_s(t,1)}.
\end{eqnarray}
\end{lemma}

\begin{proof} Equation~\eqref{eq:reduce-to-star} is readily obtained by applying Lemma~\ref{lem:symmetry-Cayley-multitype} repeatedly (in the spirit of what was done in Equation~\eqref{eq:repeat-to-star}). This implies~\eqref{eq:reduce-to-starGF} since extracting the coefficient of $\prod_{s,t\in[d],i\in[n_t]}x_{s,t,i}^{\ch_s(t,i)}$ in the left-hand side of~\eqref{eq:reduce-to-starGF} gives $\mT_{\rho,\ga}$, while extracting this coefficient in the right-hand side gives
$$\sum_{T\in \mT_{\rho,\ga^*}} \prod_{s,t\in[d]}{\ga^*_{s,t,1} \choose \ga_{s,t,1},\ldots,\ga_{s,t,n_t}}.$$
\end{proof}

The following lemma completes the proof of Theorem~\ref{thm:multitype-Cayley}.
\begin{lemma}\label{lem:forstars}
Let $\nn=(n_1,n_2,\ldots,n_d)$ be a tuple of positive integers. The generating function of star trees of profile $\nn$ is given by
\begin{eqnarray}\label{eq:GF-star-trees}
\sum_{T\emph{ star tree in } \mT_\rho(\nn)~}\prod_{s,t\in[d]} y_{s,t}^{\ch_s(t,1)}=\prod_{s\in[d]}\Bigg(\sum_{t\in[d]}y_{s,t}\Bigg)^{n_s-1}\times \sum_{A\in \Cayd} ~\prod_{(s,t)\in A}y_{s,t}.
\end{eqnarray}
\end{lemma}

\begin{proof} 
Let $T$ be a star tree in $\mT_\rho(\nn)$. Since the vertices of $T$ with labels distinct from $1$ are leaves, removing these vertices gives a tree $A$ with vertex set $\{(t,1) \mid t\in[d]\}$ and root vertex $(\rho,1)$. We call $A$ the \emph{core} of the star tree $T$ and observe that the cores identify with the Cayley trees on $[d]$. Now, any star tree in $\mT_\rho(\nn)$ is obtained by choosing a core~$A$, and then adding the leaves $(s,i)$ for $s\in[d]$ and $i\in [n_s]-\{1\}$. These leaves can be glued to any of the vertices $(1,1),(2,1),\ldots,(d,1)$ of $A$, and gluing a leaf to $(t,1)$ increases $\ch_s(t,1)$ by 1. This gives~\eqref{eq:GF-star-trees}, where $\ds \sum_{A\in \Cayd} ~\prod_{(s,t)\in A}y_{s,t}$ accounts for choosing the core $A$, and $\ds \prod_{s\in[d]}\Big(\sum_{t\in[d]}y_{s,t}\Big)^{n_s-1}$ accounts for adding the leaves.
\end{proof}

By Lemma~\ref{lem:tostars}, the generating function of multitype Cayley trees is obtained by substituting $y_{s,t}$ by $\sum_{i\in[n_t]} x_{s,t,i}$ in~\eqref{eq:GF-star-trees}. This completes the proof of Theorem~\ref{thm:multitype-Cayley}.\\


\section{Some applications}\label{sec:applications}
In this section we highlight a few consequences of Theorem~\ref{thm:multitype-Cayley}. Throughout the section, $\nn=(n_1,\ldots,n_d)$ is a fixed tuple of positive integers, and $\rho$ is in $[d]$.  We say that an edge of a tree $T\in\mT_\rho(\nn)$ has \emph{type} $(s,t)$ if it goes from a vertex of type $s$ to a vertex of type $t$. Here is a summary of the results obtained below:
\begin{compactitem}
\item In Proposition~\ref{prop:trees-by-number-edges} we recover a result of Bousquet,  Chauve, Labelle and Leroux~\cite[Proposition 2]{Leroux:multivariate-Lagrange} by enumerating trees in $\mT_{\rho}(\nn)$ with prescribed number $m_{s,t}$ of edges of type $(s,t)$.
\item  In  Proposition~\ref{prop:embedded} we recover a result by Knuth~\cite{Knuth:enumeration-trees} by counting trees in $\mT_\rho(\nn)$ such that the edges all have a type in a prescribed set $D\subset [d]^2$. These are called \emph{$D$-embedded trees}.
\item In Proposition~\ref{prop:injective-trees} we enumerate \emph{injective trees} in $\mT_{\rho}(\nn)$, that is, trees such that every vertex has at most one child of each type. Our formula generalizes a result by  Bousquet-Mélou and Chapuy~\cite[Theorem 6]{MBM-Chapuy:vertical-profile-trees}.
\item In Proposition~\ref{prop:injective-embedded-trees}, we enumerate injective $D$-embedded trees in $\mT_{\rho}(\nn)$ .
\item In Proposition~\ref{prop:number-trees-N} we enumerate trees in $\mT_{\rho}(\nn)$ with prescribed number $N_{t,\cc}$ of vertices of type $t$ having {\em indegree vector} $\cc=(c_1,\ldots,c_d)$ (that is, having $c_s$ children of type $s$). Our formula generalizes~\cite[Theorem 7]{MBM-Chapuy:vertical-profile-trees}.
\item In Proposition~\ref{prop:complete-type} we enumerate trees in $\mT_{\rho}(\nn)$ with prescribed number $N_{t,u,\cc}$ of vertices of type $t$ having a parent of type $u$ and indegree vector $\cc$. Our formula answers a question raised in~\cite{MBM-Chapuy:vertical-profile-trees}.\\
\end{compactitem}
We mention that a version of these results could be obtained for rooted forests using Corollary~\ref{cor:multitype-forest}.

Throughout the section, we denote by $\Cayd$ the set of unitype Cayley trees with vertex set $[d]$ rooted at vertex~$\rho$ and considered as oriented toward their root vertex, and for a tree $A\in\Cayd$ the notation $(s,t)\in A$ indicates that the oriented edge $(s,t)$ belongs to the oriented tree $A$.  We first count multitype Cayley trees according to the number of edges of each type. For a tuple $\mm=(m_{s,t})_{s,t\in[d]}$ of non-negative integers, we denote by $\mT_{\rho}(\mm,\nn)$ the set of trees in $\mT_{\rho}(\nn)$ having $m_{s,t}$ edges of type $(s,t)$. In order for $\mT_{\rho}(\mm,\nn)$ to be non-empty we must impose 
\begin{eqnarray}\label{eq:compatible}
\textrm{for all } s\in[d],\quad n_s=\de_{s,\rho}+\sum_{t\in[d]}m_{s,t},
\end{eqnarray}
where $\de_{s,\rho}$ denotes the Kronecker delta. The following result was obtained by Bousquet, Chauve, Labelle and Leroux in~\cite[Proposition 2]{Leroux:multivariate-Lagrange}.
\begin{prop}[\cite{Leroux:multivariate-Lagrange}]\label{prop:trees-by-number-edges}
Let $\mm=(m_{s,t})_{s,t\in[d]}$ be a tuple of non-negative integers. The number of trees in $\mT_\rho(\nn)$ with $m_{s,t}$ edges of type $(s,t)$ for all $s,t\in [d]$ is 
\begin{equation}\label{eq:trees-by-number-edges}
|\mT_{\rho}(\mm,\nn)|=\Bigg(\prod_{s,t\in[d]}n_t^{m_{s,t}}\Bigg)\Bigg( \frac{\prod_{s\in[d]}(n_s-1)!}{\prod_{s,t\in[d]}m_{s,t}!}\Bigg)\times\sum_{A\in \Cayd~}\prod_{(s,t)\in A}m_{s,t}
\end{equation}
if~\eqref{eq:compatible} holds and $|\mT_{\rho}(\mm,\nn)|=0$ otherwise.
\end{prop} 
\noindent \textbf{Remark.} 
By the matrix-tree theorem, the sum over Cayley trees in~\eqref{eq:trees-by-number-edges} can be expressed as a determinant. More precisely, this sum is the determinant of the matrix obtained by deleting the $\rho$th row and $\rho$th column of the $d\times d$ matrix $L$ having entries $\ell_{s,s}=\sum_{t\in[d],t\neq s}m_{s,t}$ for all $s\in[d]$, and $\ell_{s,t}=-m_{s,t}$ for all $t\neq s\in[d]$.
Similar determinantal expressions exist for the formulas given in Propositions \ref{prop:embedded} to \ref{prop:complete-type} and are omitted.\\

\begin{proof}
Setting $x_{s,t,i}=x_{s,t}$ for all $i$ in~\eqref{eq:GF-multitype-Cayley} gives
$$\sum_{T\in\mT_{\rho}(\nn)} \prod_{s,t \in [d]} x_{s,t}^{\#\textrm{edges of type }(s,t)}=\prod_{s\in[d]}\bigg(\sum_{t\in [d]}n_t x_{s,t}\bigg)^{n_s-1}\times \sum_{A\in\Cayd~} \prod_{(s,t)\in A}n_t x_{s,t},$$
and extracting the coefficient of $ \prod_{s,t\in [d]}x_{s,t}^{m_{s,t}}$ in this equation gives 
$$|\mT_{\rho}(\mm,\nn)|=\Bigg(\prod_{s,t\in[d]}n_t^{m_{s,t}}\Bigg)\sum_{A\in \Cayd~}\prod_{s\in[d]}\Bigg[\prod_{t\in[d]}x_{s,t}^{m_{s,t}}\Bigg]\bigg(\sum_{t\in [d]}n_t x_{s,t}\bigg)^{n_s-1}\prod_{t,(s,t)\in A}x_{s,t},$$
where the bracket notation means extraction of coefficients. Moreover, for any $s\in[d]$ there is at most one $t$ with $(s,t)\in A$ (this is the parent of $s$ in $A$) so that 
$$\Bigg[\prod_{t\in[d]}x_{s,t}^{m_{s,t}}\Bigg]\bigg(\sum_{t\in [d]}n_t x_{s,t}\bigg)^{n_s-1}\prod_{t,(s,t)\in A}x_{s,t}=\frac{\prod_{s\in[d]}(n_s-1)!}{\prod_{s,t\in[d]}m_{s,t}!}\prod_{t,(s,t)\in A} m_{s,t},$$
provided \eqref{eq:compatible} holds. This gives~\eqref{eq:trees-by-number-edges}.
\end{proof}


Next, we count \emph{embedded trees}. Given a subset $D$ of $[d]^2$, we say that a multitype Cayley tree $T$ is \emph{embedded} in~$D$ if the type of every edge of $T$ belongs to~$D$. The reason for this terminology is that $D$ can be seen as a digraph with vertex set $[d]$ and if the vertices of $T$ are sent to the vertices of $D$ corresponding to their type (i.e., $(t,i) \mapsto t$), then the edges of $T$ are sent to edges of $D$. Let $\mT_{\rho,D}(\nn)$ be the set of trees in $\mT_\rho(\nn)$ embedded in $D$. The result proved by Knuth in~\cite{Knuth:enumeration-trees} (using a variation of Joyal's endofunction technique~\cite{Joyal:Cayley}) is the following enumeration of $\mT_{\rho,D}(\nn)$.
\begin{prop}[\cite{Knuth:enumeration-trees}] \label{prop:embedded}
Let $D\subseteq [d]^2$. The number of trees in $\mT_\rho(\nn)$ embedded in $D$ is 
\begin{equation*}
|\mT_{\rho,D}(\nn)|=\prod_{s\in[d]}\Bigg(\sum_{(s,t)\in D}n_t\Bigg)^{n_s-1}\times \sum_{A\in\Cayd,~ A\subseteq D~}\prod_{(s,t)\in A}n_t.
\end{equation*} 
\end{prop}
\begin{proof}
Setting $x_{s,t,i}=1$ if $(s,t)\in D$ and $x_{s,t,i}=0$ otherwise in~\eqref{eq:GF-multitype-Cayley} directly gives the result.
\end{proof}
 
Next, we count {\em injective trees} and {\em embedded injective trees}. A multitype Cayley tree $T$ is said to be \emph{injective} if every vertex has at most one child of each type. Injective trees were introduced by Bousquet-Mélou and Chapuy in~\cite{MBM-Chapuy:vertical-profile-trees} in order to study the so-called \emph{vertical profile} of trees. Let $\mT^{\inj}_{\rho}(\mm,\nn)$ (resp. $\mT^{\inj}_{\rho,D}(\nn)$) be the subset of injective trees in $\mT_{\rho}(\mm,\nn)$ (resp. $\mT_{\rho,D}(\nn)$).
The following result generalizes~\cite[Theorem 6]{MBM-Chapuy:vertical-profile-trees}. 
\begin{prop} \label{prop:injective-trees}
Let $\mm=(m_{s,t})_{s,t\in[d]}$ be a tuple of non-negative integers. The number of injective trees in $\mT_\rho(\nn)$ with $m_{s,t}$ edges of type $(s,t)$ for all $s,t\in [d]$ is 
\begin{equation*}
|\mT^{\inj}_{\rho}(\mm,\nn)|=\prod_{s,t\in[d]} {n_t \choose m_{s,t}}\,\, \prod_{s\in[d]} (n_s-1)! \times\sum_{A\in\Cayd~}\prod_{(s,t)\in A}m_{s,t}.
\end{equation*}
if~\eqref{eq:compatible} holds and $|\mT^{\inj}_{\rho}(\mm,\nn)|=0$ otherwise.
\end{prop}
\begin{proof}
In order to choose a tree in $\mT^{\inj}_{\rho}(\mm,\nn)$, one must first choose for all $s,t\in[d]$, the set $M_{s,t}\subseteq [n_t]$ of labels of the $m_{s,t}$ vertices of type $t$ having a child of type $s$. There are $\prod_{s,t\in[d]} {n_t \choose m_{s,t}}$ such choices. Moreover, the number of trees in $\mT^{\inj}_{\rho}(\mm,\nn)$ corresponding to a given choice $(M_{s,t})_{s,t\in[d]}$ of labels is the coefficient of 
$\prod_{s,t\in[d],i\in M_{s,t}}x_{s,t,i}$ in~\eqref{eq:GF-multitype-Cayley} which is easily seen to be
$\ds\prod_{s\in[d]} (n_s-1)! \times\sum_{A\in\Cayd~}\prod_{(s,t)\in A}m_{s,t}$, provided~\eqref{eq:compatible} holds.
\end{proof}

The following result generalizes~\cite[Theorem 4]{MBM-Chapuy:vertical-profile-trees}.
\begin{prop} \label{prop:injective-embedded-trees}
Let $D\subseteq [d]^2$. The number of injective trees in $\mT_\rho(\nn)$ embedded in $D$ is 
\begin{equation*}
|\mT^{\inj}_{\rho,D}(\nn)|=\prod_{s\in[d]}{\delta_{s,\rho}-1+\sum_{t\in D(s)}n_t\choose n_s-1}(n_s-1)!\times\sum_{A\subseteq D \textrm{ in }\Cayd~}\prod_{(s,t)\in A}n_t.
\end{equation*}
where $D(s)=\{t\in [d]\mid (s,t)\in D\}$. 
\end{prop}
\begin{proof}
By~\eqref{eq:GF-multitype-Cayley}, $|\mT^{\inj}_{\rho,D}(\nn)|$ is the number of square-free monomials in the expansion of 
 $$ \prod_{s\in[d]}\bigg(\sum_{t \in D(s),i\in[n_t]}x_{s,t,i}\bigg)^{n_s-1}\times\sum_{A\subseteq D \textrm{ in }\Cayd~} ~\prod_{(s,t)\in A~}\bigg(\sum_{i\in [n_t]}x_{s,t,i}\bigg).$$
Expanding the sum over $A$ gives a sum of 
$$\sum_{A\subseteq D \textrm{ in }\Cayd~}\prod_{(s,t)\in A}n_t$$ 
square-free monomials of the form $ \prod_{s\in[d]\setminus \{\rho\}} y_s$, with $y_s$ in $X_s=\{x_{s,t,i}\}_{t\in D(s),i\in[n_t]}$. Moreover, there are clearly
$$\prod_{s\in[d]}{\delta_{s,\rho}-1+\sum_{t\in D(s)}n_t\choose n_s-1}(n_s-1)!$$
square-free monomials not containing the variable $y_s$ for all $s\in[d]\setminus \{\rho\}$ in the expansion of 
$$\prod_{s\in[d]}\bigg(\sum_{t\in D(s),i\in[n_t]}x_{s,t,i}\bigg)^{n_s-1}=\prod_{s\in[d]}\bigg(\sum_{y\in X_s}y\bigg)^{n_s-1}.$$ 
\end{proof}

We will now count trees according to their \emph{indegree vectors} and \emph{complete degree vectors}. Recall that for a tuple $\ga=(\ga_{s,t,i})_{s,t\in [d],i\in[n_t]}$ of non-negative integers, $\mT_{\rho,\ga}$ denotes the set of trees in $\mT_\rho(\nn)$ such that $\ch_s(t,i)=\ga_{s,t,i}$ for all $s,t\in[d],i\in[n_t]$. 
\begin{prop} \label{prop:nb-multitype-Cayley}
Let $\ga=(\ga_{s,t,i})_{s,t\in [d],i\in[n_t]}$ be a tuple of non-negative integers, and let $m_{s,t}=\sum_{i\in[n_t]}\ga_{s,t,i}$. The number of trees with indegree vector $\ga$ (hence having $m_{s,t}$ edges of type $(s,t)$) is 
\begin{equation}\label{eq:nb-multitype-Cayley}
|\mT_{\rho,\ga}|=\frac{ \prod_{t\in [d]}(n_t-1)!}{ \prod_{s,t\in[d],i\in[n_t]}\ga_{s,t,i}!}\times \sum_{A\in\Cayd} ~\prod_{(s,t)\in A}m_{s,t},
\end{equation}
if~\eqref{eq:compatible} holds, and $|\mT_{\rho,\ga}|=0$ otherwise.
\end{prop}

\begin{proof} 
By extracting the coefficient of $\prod_{s,t \in [d],i\in[n_t]}x_{s,t,i}^{\ga_{s,t,i}}$ in~\eqref{eq:GF-multitype-Cayley} we get
\begin{eqnarray}
|\mT_{\rho,\ga}|&=&\Bigg[\prod_{s,t \in [d],i\in[n_t]}\!\!\!x_{s,t,i}^{\ga_{s,t,i}}\Bigg]\sum_{A\in\Cayd}\Bigg(\prod_{(s,t)\in A}\,\sum_{i\in[n_{t}]}\!\!x_{s,t,i}\Bigg)\prod_{s\in[d]}\Bigg(\sum_{t\in[d],i\in[n_t]}\!\!\!x_{s,t,i}\Bigg)^{n_{s}-1}\nonumber \\
&=&\sum_{A\in\Cayd}\prod_{s\in [d]} C_{A,s} \label{eq:cardTga}
\end{eqnarray}
where 
$$
C_{A,s}=\Bigg[\prod_{t \in [d],i\in[n_t]}\!\!\!x_{s,t,i}^{\ga_{s,t,i}}\Bigg] \Bigg(\prod_{t,(s,t)\in A}\sum_{i\in[n_{t}]}x_{s,t,i}\Bigg)\Bigg(\sum_{t\in[d],i\in[n_t]}\!\!\!x_{s,t,i}\Bigg)^{n_{s}-1}.\\
$$
We now assume that~\eqref{eq:compatible} holds. Since there is no $t\in[d]$ such that $(\rho,t)\in A$, 
$$C_{A,\rho}\,=\,\Bigg[\prod_{t \in [d],i\in[n_t]}\!\!\!x_{\rho,t,i}^{\ga_{\rho,t,i}}\Bigg] \Bigg(\sum_{t\in[d],i\in[n_t]}\!\!\!x_{\rho,t,i}\Bigg)^{n_{\rho}-1}\,=\,\frac{(n_\rho-1)!}{\prod_{t \in [d],i\in[n_t]}\ga_{\rho,t,i}!}.$$
Similarly, by denoting $s'$ the parent of a vertex $s\neq \rho$ in $A$, we get
\begin{eqnarray*}
C_{A,s}=\Bigg[\prod_{t \in [d],i\in[n_t]}\!\!\!x_{s,t,i}^{\ga_{s,t,i}}\Bigg]\sum_{i\in[n_{s'}]} x_{s,s',i}\Bigg(\sum_{t\in[d],i\in[n_t]}\!\!\!x_{s,t,i}\Bigg)^{n_{s}-1}= \Bigg(\sum_{i\in[n_{s'}]}\!\!\!\ga_{s,s',i}\Bigg)\frac{(n_s-1)!}{\prod_{t \in [d],i\in[n_t]}\ga_{s,t,i}!}.
\end{eqnarray*}
Using $\ds \sum_{i\in[n_{s'}]}\ga_{s,s',i}=m_{s,s'}$, and plugging the expression of $C_{A,s}$ in~\eqref{eq:cardTga} gives~\eqref{eq:nb-multitype-Cayley}.
\end{proof}

Another way of stating~\eqref{eq:nb-multitype-Cayley} is by fixing the number of vertices of each indegree type (but without fixing their labels). We say that a vertex of a tree $T\in \mT_{\rho}(\nn)$ has \emph{indegree type} $\cc=(c_1,\ldots,c_d)$ if it has $c_s$ children of type $s$ for all $s\in[d]$. For instance, the vertex $(1,4)$ in Figure~\ref{fig:example-multitype-Cayley} has indegree type $\cc=(1,0,2)$.
Let $\CC=[n_1]\times\ldots \times[n_d]$, and let $\NN=(N_{t,\cc})_{t\in[d],\cc\in\CC}$ be a tuple of non-negative integers. Let $\mT_{\rho}^\bN$ be the set of trees in $\mT_\rho(\nn)$ having $N_{t,\cc}$ vertices of type~$t$ with indegree type~$\cc$, for all $t\in[d],\cc\in\CC$. The following result generalizes~\cite[Theorem 7]{MBM-Chapuy:vertical-profile-trees}. 
\begin{prop}\label{prop:number-trees-N}
Let $\NN=(N_{t,\cc})_{t\in[d],\cc\in\CC}$ be a tuple of non-negative integers. Let $\ds n_t=\sum_{\cc\in \CC}N_{t,\cc}$ for all $t\in[d]$, let $\ds m_{s,t}=\sum_{\cc\in\CC}c_sN_{t,\cc}$, and let $\ds N(k)=\sum_{s,t\in[d],\cc\in\CC, \atop \textrm{ such that } c_s=k}N_{t,\cc}$. 
The number of trees in $\mT_\rho(\nn)$ having $N_{t,\cc}$ vertices of type~$t$ with indegree type~$\cc$ is
\begin{equation*}
|\mT_{\rho}^\bN|= \frac{\ds \prod_{t\in[d]}n_t!(n_t-1)!}{\ds \prod_{t\in[d],\cc\in\CC}N_{t,\cc}! ~ \prod_{k\geq 0}k!^{N(k)}}\times \sum_{A\in \Cayd} ~\prod_{(s,t)\in A}m_{s,t}, 
\end{equation*}
if~\eqref{eq:compatible} holds and $|\mT_{\rho}^\bN|=0$ otherwise.
\end{prop}
\begin{proof}
By definition, $\mT_{\rho}^\bN=\bigcup_{\ga\in\Ga(\NN)}\mT_{\rho,\ga}$, where $\Ga(\NN)$ is the set of tuples on non-negative integers $(\ga_{s,t,i})_{s,t\in [d],i\in[n_t]}$ such that for all $t\in[d]$ and all $\cc\in \CC$, there are $N_{t,\cc}$ integers $i\in[n_t]$ with $(\ga_{1,t,i},\ldots,\ga_{d,t,i})=\cc$. 
Since $\ds |\Ga(\NN)|=\frac{\prod_{t\in[d]} n_t!}{\prod_{t\in[d],\cc\in\CC}N_{t,\cc}!}$, using~\eqref{eq:nb-multitype-Cayley} gives the result.
\end{proof}
\noindent \textbf{Remark.} Propositions~\ref{prop:trees-by-number-edges} to~\ref{prop:number-trees-N} give particularly nice counting formulas (which factorize completely) when there is a unique tree $A\in \Cayd$ contributing to the sum (the cases considered in~\cite{MBM-Chapuy:vertical-profile-trees} are all of this form). It is therefore interesting to understand when this favorable situation occurs. A quick investigation reveals that it occurs when the type of the edges allowed to appear in the trees to be enumerated belong to a set $D\subseteq[d]^2$ which can be partitioned as $D=A\cup A'$, where $A$ is  the edge set of any tree in $\Cayd$ and $A'$ is a set of edges of the form $(s,t)$ where $t=s$ or $t$ is a descendant of $s$ in $A$ (i.e., $(t,s)$ is in $A$). For instance, if $A\in \Cayd$, and $m_{s,t}=0$ unless $(s,t)$ or $(t,s)$ is in $A$, then Proposition~\ref{prop:number-trees-N} gives
$$|\mT_{\rho}^\bN|= \frac{\prod_{t\in[d]}n_t!(n_t-1)!}{\prod_{t\in[d],\cc\in\CC}N_{t,\cc}! ~ \prod_{k\geq 0}k!^{N(k)}}\prod_{(s,t)\in A}m_{s,t}. $$

As this section's last application of Theorem~\ref{thm:multitype-Cayley}, we count trees according to their complete degree type. 
We say that a vertex of a tree in $\mT_\rho(\nn)$ has \emph{complete degree type} $(t,u,\cc)$ if it has type~$t$, indegree type~$\cc$ and its parent has type~$u$, with the convention that the fictitious ``parent'' of the root vertex has type $d+1$. For instance, the vertex $(1,4)$ in Figure~\ref{fig:example-multitype-Cayley} has complete degree type $(1,2,(1,0,2))$, while the root vertex has complete degree type $(2,4,(1,0,1))$. Let $\vect{\mT}_{\rho}^\NN$ be the set of multitype Cayley trees with a root vertex of type $\rho$ having $N_{t,u,\cc}$ vertices of complete degree type $(t,u,\cc)$ for all $t\in [d],u\in[d+1],\cc\in\CC$. 
\begin{prop}\label{prop:complete-type}
Let $\NN=(N_{t,u,\cc})_{t\in[d], u\in[d+1],\cc\in\CC}$ be a tuple of non-negative integers. 
Let $\ds ~m_{t,u}=\sum_{\cc\in\CC}N_{t,u,\cc}$,
let $\ds ~m_{s,t,u}=\!\sum_{\cc\in\CC}\!c_sN_{t,u,\cc}$, let $\ds n_t=\sum_{u\in [d+1]}m_{t,u}$, and let
$$\ds ~N(k)=\!\!\sum_{s,t\in[d],u\in[d+1],\cc\in\CC,\atop \textrm{ such that } c_s=k}\!\!N_{t,u,\cc}.$$
The number of trees in $\mT_\rho(\nn)$ having $N_{t,u,\cc}$ vertices of complete degree type~$(t,u,\cc)$ is
\begin{eqnarray}\label{eq:nb-complete-multitype}
|\vect{\mT}_{\rho}^\NN|= \frac{\displaystyle \prod_{t\in[d]}n_t!\, \prod_{s,t\in[d],\, m_{s,t}>0}(m_{s,t}-1)!}{\displaystyle \prod_{t,u\in[d],\cc\in\CC}\!\!\!\!\!\!N_{t,u,\cc}! ~ \prod_{k\geq 0} k!^{N(k)}}\times \sum_{A\subseteq G(\NN)} ~\prod_{\big((s,t),(t,u)\big)\in A}m_{s,t,u}\, , 
\end{eqnarray}
provided that $m_{s,d+1}=\de_{s,\rho}$, and $m_{s,t}=\sum_{u\in[d+1]}m_{s,t,u}$ for all $s,t\in[d]$ ($|\vect{\mT}_{\rho}^\NN|=0$ otherwise),
where $G(\NN)$ is the graph with vertex set $V=\{(t,u)\in[d]\times[d+1],m_{t,u}>0\}$ and edge set $\{\big((s,t),(t,u)\big) \mid (s,t),(t,u)\in V\}$, and the sum is taken over the spanning trees $A$ of $G(\NN)$ oriented toward the vertex $(\rho,d+1)$.
\end{prop}
As an illustration of Proposition~\ref{prop:complete-type}, consider the case $d=1$, $n_1=n$ and $m_{1,1}=n-1>0$. 
In this case, the graph $G(\NN)$ has two vertices $(1,1)$ and $(1,2)$ and two edges: one edge going from $(1,1)$ to $(1,2)$ and one edge going from $(1,1)$ to $(1,1)$. Hence the sum in~\eqref{eq:nb-complete-multitype} reduces to $m_{1,1,2}$ which is the specified degree of the root vertex. 
Equation \eqref{eq:nb-complete-multitype} then gives the number $|\vect{\mT}_{\rho}^\NN|$ of Cayley trees with $n$ vertices, with $N_c$ non-root vertices of indegree $c$ and a root vertex of degree $\ell$ as
$$|\vect{\mT}_{\rho}^\NN|=\frac{n!(n-2)!}{\prod_{c\geq 0}N_c!c!^{N_c}(\ell-1)!}.$$
\begin{proof}
The proof of Proposition \ref{prop:complete-type} is based on the observation that counting trees according to complete degree types can be seen as a special case of counting trees according to indegree types. Indeed, let us define the \emph{complete type} of a vertex $v$ of a tree $T\in \mT_{\rho}(\nn)$ as the pair $(t,u)\in[d]\times[d+1]$, where $t$ is the type of $v$ and $u$ is the type of the parent of $v$, with the convention that $u=d+1$ if $v$ is the root vertex. We can then directly apply Proposition~\ref{prop:number-trees-N} in order to count trees according to their complete degree type (upon observing that the number of vertices of complete type $(s,t)$ is $m_{s,t}$, and the number of edges from a vertex of complete type $(s,t)$ to a vertex of complete type $(t,u)$ is $m_{s,t,u}$). 
We only need to remember that the vertices of our Cayley trees are still labeled within their ``original type'' and not within their complete type (that is, the labels of vertices of complete type $(s,t)$ have a label in $[n_t]$ and not in $[m_{s,t}]$), which is accounted for by a factor $\ds \frac{\prod_{t\in[d]}n_t!}{\prod_{s,t\in[d]}m_{s,t}!}$ in~\eqref{eq:nb-complete-multitype}. 
\end{proof}

Observe that using the same techniques as for Proposition~\ref{prop:complete-type} one could enumerate multitype Cayley trees according to the number of vertices having given indegree type, given type, with parent of given type, and grandparent of given type, etc. 
We now investigate a case where the sum in~\eqref{eq:nb-complete-multitype} simplifies greatly. The following result answers a question raised in~\cite[Section 8.2]{MBM-Chapuy:vertical-profile-trees}.
\begin{cor}\label{cor:complete-type-special}
Suppose, using the notation of Proposition~\ref{prop:complete-type}, that $\rho=d$ and all the pairs $(s,t)\in [d]\times [d+1]$ such that $m_{s,t}>0$ satisfy $t\leq s+1$. In this case,
\begin{equation}\label{eq:nb-complete-multitype-specialcase}
\begin{split}
|\vect{\mT}_{d}^\NN|&=\frac{\displaystyle \prod_{t\in[d]}n_t! \prod_{s,t\in[d],\, m_{s,t}>0}(m_{s,t}-1)!}{\displaystyle \prod_{t,u\in[d],\cc\in\CC}\!\!\!\!\!\!N_{t,u,\cc}! ~ \prod_{k\geq 0} k!^{N(k)}} \times \prod_{s,t\in[d], \textrm{ such that }\atop t<s \textrm{ and } m_{s,t}>0 }\!\!\! m_{s,t} \\
&\quad\quad\quad\quad\quad\quad\quad\quad\quad\quad\quad\times \mu_1\prod_{s=2}^d \left(m_{s-1,s,s+1}\,\mu_{s}+m_{s-1,s,s}m_{s,s,s+1}\right),
\end{split}
\end{equation}
where $\mu_s=m_{s,s}-m_{s,s,s}$ if $m_{s,s}>0$ and $\mu_s=1$ if $m_{s,s}=0$. 
\end{cor}
The result given in~\cite[Theorem 8]{MBM-Chapuy:vertical-profile-trees} corresponds to~\eqref{eq:nb-complete-multitype-specialcase} in the case where $m_{s,s}=0$ for all $s$ (in this case, $\mu_{s}=1$ and $m_{s-1,s,s}=0$). We now briefly explain how to derive Corollary~\ref{cor:complete-type-special} from Proposition~\ref{prop:complete-type}. Let $G'$ be the directed graph with vertex set $V'=\{(s,t)\in[d]\times[d+1]\mid t\leq s+1\}$ and edge set $E'=\{((s,t),(t,u)) \mid s,t\in[d],u\in[d+1]\}$. We denote $e_{s,t,u}$ the edge $((s,t),(t,u))\in E'$. By hypothesis, the graph $G(\NN)$ is a subgraph of $G'$. 
\begin{claim} \label{claim:characterization-of-Atrees}
Let $A$ be a subset of $E'$, such that $A$ contains exactly one edge going out of each vertex $(s,t)\in[d]^2$ (that is, an edge of the form $e_{s,t,u}$), but contains no loop (that is, no edge of the form $e_{s,s,s}$).
 Then $A$ is a spanning tree of $G'$ oriented toward the root vertex $(d,d+1)$ if and only if for all $s\in\{2,\ldots,d\}$, $A$ contains either the edge $e_{s-1,s,s+1}$ or both the edges $e_{s-1,s,s}$ and $e_{s,s,s+1}$. 
\end{claim}
\begin{proof}
We consider the \emph{lexicographic order} on the set $V'$ of vertices (in this order $(s,t)\leq_{\lex}(s',t')$ if and only if $s< s'$ or $s=s'$ and $t\leq t'$). Observe that the root vertex $(d,d+1)$ is the largest vertex in the lexicographic order and that $A$ is a spanning tree if and only if the unique directed path in $A$ starting from any vertex $(s,t)\in[d]^2$ leads to the root vertex $(d,d+1)$. 
Suppose first that $A$ is a spanning tree of $G'$ oriented toward the root vertex $(d,d+1)$. Since $e_{s-1,s,s+1}$ and $e_{s-1,s,s}$ are the only edges in $E'$ going from a vertex $(t,u)\leq_{\lex}(s-1,s)$ to a vertex $(u,v)>_{\lex}(s-1,s)$, we know that one of these edges is in $A$. Similarly, since $e_{s-1,s,s+1}$ and $e_{s,s,s+1}$ are the only edges in $E'$ going from a vertex $(t,u)\leq_{\lex}(s,s)$ to a vertex $(u,v)>_{\lex}(s,s)$, we know that one of these edges is in $A$.
Suppose conversely that for all $s\in\{2,\ldots,d\}$, $A$ contains either the edge $e_{s-1,s,s+1}$ or both the edges $e_{s-1,s,s}$ and $e_{s,s,s+1}$. We want to prove that for each $(s,t)\in [d]^2$ the directed path in $A$ starting from $(s,t)$ leads to the root vertex $(d,d+1)$. This property is obvious in the case $t=s+1$. Moreover in the case $t\leq s$ we observe that the vertices decrease strictly for the lexicographic order along the directed path starting at $(s,t)$ until a vertex of the form $(s,s+1)$ is reached. This concludes the proof.
\end{proof}
\noindent We now complete the proof of~Corollary~\ref{cor:complete-type-special}. Using Claim~\ref{claim:characterization-of-Atrees}, one can decompose the sum in~\eqref{eq:nb-complete-multitype} according to the fact that the spanning tree $A$ of $G(\NN)$ contains the edge $e_{s-1,s,s+1}$ or both the edges $e_{s-1,s,s}$ and $e_{s,s,s+1}$. We now claim that in the rightmost product of~\eqref{eq:nb-complete-multitype-specialcase} the case $e_{s-1,s,s+1}\in A$ is accounted for by the term $m_{s-1,s,s+1}\,\mu_{s}$, while the case $\{e_{s-1,s,s},e_{s,s,s+1}\}\subseteq A$ is accounted for by the term $m_{s-1,s,s}m_{s,s,s+1}$. We leave the details of this statement to the reader, and only mention that the term $\mu_s$ accounts for the choice of the edge of $A$ going out of the vertex $(s,s)$ (in the case $(s,s)\in V$), while the term $m_{s,t}$ in the first product accounts for the choice of the edge of $A$ going out of the vertex $(s,t)$ (in the case $(s,t)\in V$ with $t<s$).\\

\section{Multivariate Lagrange Inversion formula}\label{sec:Lagrange}
In this section we show that Theorem~\ref{thm:multitype-Cayley}, via Proposition~\ref{prop:nb-multitype-Cayley}, implies the multivariate Lagrange inversion formula~\cite{Good:multivariate-Lagrange}. 
There are several versions of this formula which have been shown to be equivalent to each other; see~\cite{Gessel:multivariate-Lagrange} for a survey. Here we will derive a version due to Bender and Richmond~\cite{Bender-Richmond:multivariate-Lagrange}.

We will consider power series in the variables $x_1,\ldots,x_d$, and we denote $\xx=(x_1,\ldots,x_d)$. For a tuple of integers $\nn=(n_1,\ldots,n_d)$, we denote $\xx^\nn=x_1^{n_1}x_2^{n_2}\cdots x_d^{n_d}$, and we denote $[\xx^\nn]f(\xx)$ for the coefficient of the monomial $\xx^\nn$ in a power series $f(\xx)$. 

\begin{thm}[multivariate Lagrange inversion formula~\cite{Bender-Richmond:multivariate-Lagrange}]\label{thm:multivariate-Lagrange}
Let $\GG_1,\ldots,\GG_{d+1},$ be power series in $d$ variables with non-zero constant terms. There exists a unique tuple $(f_1,\ldots,f_d)$ of power series in $x_1,\ldots,x_d$ satisfying
\begin{eqnarray}\label{eq:Lagrange-setting}
f_t(x_1,\ldots,x_d)=x_t\GG_t(f_1,\ldots,f_d)
\end{eqnarray}
for all $t\in[d]$. Moreover, for any tuple $\nn=(n_1,\ldots,n_d)$ of positive integers
\begin{eqnarray}\nonumber
[\xx^\nn]\GG_{d+1}(f_1,\ldots,f_d)&\!=\!&[\xx^\nn] \Bigg(\prod_{t\in[d]} \frac{x_t}{n_t}\Bigg)\Bigg(\sum_{A\in\Caydp}\,\prod_{t\in[d+1]}\!\!\Bigg(\prod_{(s,t)\in A}\frac{\partial}{\partial x_s }\Bigg)\GG_t(\xx)\Bigg),
\end{eqnarray}
where $\Caydp$ is the set of unitype Cayley trees with vertex set $[d+1]$ and root vertex $d+1$, considered as oriented toward their root vertex.
\end{thm}

Observe that the sum over $\Caydp$ in Theorem \ref{thm:multivariate-Lagrange} could be written as a determinant of differential operators using the matrix-tree theorem. In~\cite{Bender-Richmond:multivariate-Lagrange} it was proved that Theorem~\ref{thm:multivariate-Lagrange} is equivalent to some more traditional forms of the multivariate Lagrange inversion formula. This formulation was actually already implicit in~\cite{Goulden-Kulkarni:multivariate-Lagrange}. A combinatorial proof of Theorem~\ref{thm:multivariate-Lagrange} was first given in~\cite{Leroux:multivariate-Lagrange}. Combinatorial proofs of other forms of the multivariate Lagrange inversion formula were given in~\cite{Ehrenborg-Mendez:multivariate-Lagrange,Gessel:multivariate-Lagrange}. 

We will now prove Theorem~\ref{thm:multivariate-Lagrange} starting from Proposition~\ref{prop:nb-multitype-Cayley}. The existence and uniqueness of power series $f_1,\ldots,f_d$ satisfying~\eqref{eq:Lagrange-setting} is clear since their coefficients can be determined inductively from these equations. We now give an interpretation of these series as generating functions of trees (which is equivalent to what is done for instance in~\cite[Section 1]{Goulden-Kulkarni:multivariate-Lagrange}).
\begin{lemma}\label{lem:interpretation-trees}
For any tuple $\nn=(n_1,\ldots,n_d)$ of non-negative integers, and for all $\rho\in[d]$, the series $f_\rho$ defined by~\eqref{eq:Lagrange-setting} satisfies
$$ [\xx^{\nn}]f_\rho=\left(\prod_{t\in[d]}\frac{1}{n_t!}\right)\Bigg(\sum_{T\in\mT_\rho(\nn)}\GG_{\ch(T)}\prod_{s,t\in[d],i\in[n_t]}\ch_s(t,i)!\Bigg),$$
where $\mT_\rho(\nn)$ is the set of multitype Cayley trees defined in Section~\ref{sec:multitype-Cayley}, and
$$\GG_{\ch(T)}=\prod_{t\in[d],i\in[n_t]}[y_1^{\ch_1(t,i)}\cdots y_d^{\ch_d(t,i)}]\,\GG_t(y_1,\ldots,y_d).$$
\end{lemma}
\begin{proof}
For all $\rho\in[d]$, we define the power series $\wf_\rho$ in the variables $x_1,\ldots,x_d$ by
$$\wf_\rho:=\sum_{n_1,\ldots,n_d\geq 0}\Bigg(\prod_{t\in[d]}\frac{x_t^{n_t}}{n_t!}\Bigg)\Bigg(\sum_{T\in\mT_\rho(\nn)}\GG_{\ch(T)}\prod_{s,t\in[d],i\in[n_t]}\ch_s(t,i)!\Bigg).$$
We want to prove that $\wf_\rho=f_\rho$. 
Let $\mU_\rho(\nn)$ be the set of rooted trees obtained by taking a tree $T$ in $\mT_\rho(\nn)$, unlabeling its vertices, and for each vertex $v$ and each type $s\in[d]$ assigning a total order to the children of $v$ of type $s$. This induces a $\left(\prod_{t\in[d]}n_t!\right)$-to-$\left(\prod_{s,t\in[d],i\in[n_t]}\ch_s(t,i)!\right)$ correspondence between the sets $\mT_\rho(\nn)$ and $\mU_\rho(\nn)$. 
Thus, 
\begin{eqnarray}
\displaystyle \wf_\rho~=~\sum_{n_1,\ldots,n_d\geq 0}\xx^\nn\sum_{T\in\mU_\rho(\nn)}\GG_{\ch(T)}~=~\sum_{T\in\mU_\rho}~ \prod_{v \textrm{ vertex of } T}\om(v), \label{eq:simplef}
\end{eqnarray}
where $\displaystyle \mU_\rho=\bigcup_{n_1,\ldots,n_d\geq 0} \mU_\rho(\nn)$, and for a vertex $v$ of type $t$, 
$$\om(v):=x_t \cdot [y_1^{\ch_1(v)}\cdots y_d^{\ch_d(v)}]\, \GG_t(y_1,\ldots,y_d).$$

We now consider the classical decomposition of trees obtained by deleting the root vertex (see for instance~\cite{Flajolet:analytic}).
Let $\mU_{\rho,(c_1,\ldots,c_d)}$ be the set of trees in $\mU_\rho$ such that the root vertex has $c_s$ children of type $s$ for all $s\in[d]$. By deleting the root vertex of the trees in $\mU_{\rho,(c_1,\ldots,c_d)}$, one gets a bijection between $\mU_{\rho,(c_1,\ldots,c_d)}$ and the Cartesian product 
${\mU_1}^{c_1}\times \cdots \times{\mU_d}^{c_d}$. 
This gives the following generating function equation 
$$\sum_{T\in\mU_{\rho,(c_1,\ldots,c_d)}}~\prod_{v \textrm{ vertex of } T}\om(v)=x_\rho\cdot \Big([y_1^{c_1}\cdots y_d^{c_d}]\,\GG_{\rho}(y_1,\ldots,y_d)\Big) \cdot \wf_1^{c_1}\cdots \wf_d^{c_d}.$$
Thus~\eqref{eq:simplef} becomes 
$$\wf_\rho\,=\,x_\rho\sum_{c_1,\ldots,c_d\geq 0}\Big([y_1^{c_1}\cdots y_d^{c_d}]\,\GG_{\rho}(y_1,\ldots,y_d)\Big)\, \wf_1^{c_1}\cdots \wf_d^{c_d}\,=\,x_\rho\GG_\rho(\wf_1,\ldots,\wf_d).$$
Since $\wf_1,\ldots,\wf_d$ are power series satisfying~\eqref{eq:Lagrange-setting}, they are equal to $f_1,\ldots,f_d$ respectively.
\end{proof}

We define $f_{d+1}=x_{d+1}\GG_{d+1}(f_1,\ldots,f_d)$, $\xx'=(x_1,\ldots,x_{d+1})$, $n_{d+1}=1$, and $\nn'=(n_1,\ldots,n_{d+1})$. With this notation, Equation~\eqref{eq:Lagrange-setting} is valid for all $t\in[d+1]$, and $[\xx^{\nn}]\GG_{d+1}(f_1,\ldots,f_d)=\big[{\xx'}^{\nn'}\big]f_{d+1}$. Hence applying Lemma~\ref{lem:interpretation-trees} gives 
\begin{equation}\label{eq:interpret-trees}
[\xx^{\nn}]\GG_{d+1}(f_1,\ldots,f_d)=\Bigg(\prod_{t\in[d]}\frac{1}{n_t!}\Bigg)\Bigg(\sum_{T\in\mT_{d+1}(\nn')}\GG_{\ch(T)}\prod_{s,t\in[d+1],i\in[n_t]}\ch_s(t,i)!\Bigg).
\end{equation}
We now partition the set of trees $\mT_{d+1}(\nn')$ according to the number $m_{s,t}$ of edges of type $(s,t)$. Let $M(\nn)$ be the set of tuples of non-negative integers $(m_{s,t})_{s,t\in[d+1]}$ such that $\sum_{t\in[d+1]}m_{s,t}=n_s$ for all $s\in[d]$, and $m_{d+1,t}=0$ for all $t\in[d+1]$. For a tuple $\mm\in M(\nn)$, let $\Ga(\mm)$ be the set of tuples $(\ga_{s,t,i})_{s,t\in[d+1],i\in [n_t]}$ of non-negative integers such that $\sum_{i\in[n_t]}\ga_{s,t,i}=m_{s,t}$ for all $s,t\in[d+1]$. Since $\bigcup_{\mm\in M(\nn)}\Ga(\mm)$ is the set of all possible indegree vectors for the trees in $\mT_{d+1}(\nn')$, we can rewrite the right-hand side of~\eqref{eq:interpret-trees} as  
$$
\textup{RHS}=\Bigg(\prod_{t\in[d]}\frac{1}{n_t!}\Bigg)\Bigg(\sum_{\mm \in M(\nn)}\sum_{\ga \in \Ga(\mm)} |\mT_{d+1,\ga}|\,\GG_\ga \prod_{s,t\in[d+1],i\in[n_t]}\ga_{s,t,i}! \Bigg),
$$
where 
$$\GG_{\ga}=\prod_{t\in[d+1],i\in[n_t]}[x_1^{\ga_{1,t,i}}\cdots x_d^{\ga_{d,t,i}}]\GG_t(\xx).$$
Using Proposition~\ref{prop:nb-multitype-Cayley} then gives
\begin{equation}\label{eq:RHS}
\textup{RHS}=\Bigg(\prod_{t\in[d]}\frac{1}{n_t}\Bigg)\Bigg(\sum_{\mm \in M(\nn)}\sum_{\ga \in \Ga(\mm)} \GG_\ga \sum_{A\in\Caydp~}\prod_{(s,t)\in A}m_{s,t}\Bigg).
\end{equation}
Now observe that by definition of $\Gamma(\mm)$, 
$$\sum_{\ga \in \Ga(\mm)} \GG_\ga =\prod_{t\in[d+1]}[x_1^{m_{1,t}}\cdots x_d^{m_{d,t}}]\GG_t(\xx)^{n_t}.$$ 
Hence for any Cayley tree $A\in\Caydp$,
$$\sum_{\ga \in \Ga(\mm)} \GG_\ga \prod_{(s,t)\in A}m_{s,t}=\prod_{t\in[d+1]}[x_1^{m_{1,t}}\cdots x_d^{m_{d,t}}]\Bigg(\prod_{(s,t)\in A}x_s\frac{\partial}{\partial x_s}\Bigg)\GG_t(\xx)^{n_t},$$
and by definition of $M(\nn)$,
$$\sum_{\mm \in M(\nn)}\sum_{\ga \in \Ga(\mm)} \GG_\ga\prod_{(s,t)\in A}m_{s,t} = [\xx^\nn]\prod_{t\in[d+1]}\Bigg(\prod_{(s,t)\in A}x_s\frac{\partial}{\partial x_s}\Bigg)\GG_t(\xx)^{n_t}.$$
Moreover, $\displaystyle \prod_{(s,t)\in A}x_s=\prod_{t\in[d]}x_t$. Thus~\eqref{eq:RHS} gives 
\begin{eqnarray*}
\textup{RHS}
&=&[\xx^\nn]\Bigg(\prod_{t\in[d]}\frac{x_t}{n_t}\Bigg)\Bigg(\sum_{A\in\Caydp~}\prod_{t\in[d+1]}\Bigg(\prod_{(s,t)\in A}\frac{\partial}{\partial x_s}\Bigg)\GG_t(\xx)^{n_t}\Bigg).
\end{eqnarray*}
This completes the proof of Theorem~\ref{thm:multivariate-Lagrange}.\\

\section{Enumeration of plane trees and cacti}\label{sec:cacti}
In this section we investigate the enumeration of plane trees and cacti using the same philosophy behind the proof of Theorem~\ref{prop:nb-multitype-Cayley}. Recall that a \emph{rooted plane tree} is a rooted tree in which the children of each vertex are ordered. A \emph{vertex-labeled} rooted plane tree is a rooted plane tree in which the $n$ vertices receive distinct labels in $[n]$; equivalently it is a rooted Cayley tree in which the children of each vertex are ordered. 

As mentioned in the introduction, the problem of counting rooted plane trees is equivalent to the problem of counting rooted Cayley trees. Let us illustrate our point by enumerating rooted (unitype) plane trees having $N_i$ vertices with $i$ children. By the case $d=1$ of~\eqref{eq:nb-multitype-Cayley}, for any tuple of non-negative integers $\ga=(\ga_1,\ldots,\ga_n)$ summing to $n-1$, the number of rooted Cayley trees with $n$ vertices in which vertex~$i$ has $\ga_i$ children for all $i\in[n]$ is
$${n-1 \choose \ga_1,\ga_2,\ldots,\ga_n}.$$ 
This means that the set $\mP_{\ga}$ of vertex-labeled rooted plane trees in which vertex $i$ has $\ga_i$ children has cardinality
\begin{equation}\label{eq:plane-gamma}
|\mP_{\ga}|=(n-1)!.
\end{equation}
Let $\NN=(N_0,N_1,\ldots)$ be a sequence of non-negative integers summing to $n$, and let $\mL^\NN$ be the set of vertex-labeled rooted plane trees having $N_i$ vertices with $i$ children for all $i\geq 0$. This set is empty unless $\sum_{i\geq 0} N_i=1+\sum_{i\geq 0}i\,N_i$. Moreover, in this case from~\eqref{eq:plane-gamma} we immediately get
$$
|\mL^\NN|=(n-1)!{n \choose N_0,N_1,N_2,\ldots}.$$
Now, since a rooted plane tree has no non-trivial automorphisms, there are $n!$ ways of labeling its $n$ vertices. Thus the number of \emph{unlabeled} rooted plane trees having $N_i$ vertices with $i$ children for all $i$ is 
\begin{equation}\label{eq:plane-N}
|\mP^\NN|=\frac{1}{n}{n \choose N_0,N_1,N_2,\ldots},
\end{equation}
if $\sum_{i\geq 0} N_i=1+\sum_{i\geq 0}i\,N_i$ and 0 otherwise. Classically, this result is obtained using \L ukasiewicz words and the Cycle Lemma (see e.g.~\cite[Section 5.3]{EC2}).

We could also have proved~\eqref{eq:plane-N} directly, that is, without using~\eqref{eq:nb-multitype-Cayley}. Indeed, by a reasoning analogous to Lemma~\ref{lem:symmetry-Cayley-multitype} one easily shows that $|\mP_{\ga}|=|\mP_{\ga'}|$ for any tuples $\ga=(\ga_1,\ldots,\ga_n)$ and $\ga'=(\ga_1',\ldots,\ga_n')$ of non-negative integers summing to $n-1$. This, in turn, immediately gives~\eqref{eq:plane-gamma}. We will not dwell further into this parallel approach for counting plane trees.

In the rest of this section we deal with the enumeration of planar cacti. We give precise definitions below, but roughly speaking, a (planar) $d$-cactus is a tree-like structure in which edges are replaced by $d$\emph{-gons} (polygons with $d$ sides).
Examples are shown in Figure~\ref{fig:example-cacti}. We shall enumerate planar $d$-cacti according to the degree distribution of their vertices. This enumeration is equivalent to the computation of certain connection coefficients in the symmetric group (corresponding to the ``minimal'' factorizations of the long cycle into $d$ permutations, see~\cite{GJ92} or~\cite[Section 1.3]{LZ}). The results we obtain are not new: they have been obtained by the Lagrange inversion formula in \cite{GJ92, BBLL}, representation theory~\cite{J}, and even combinatorial methods~\cite{Leroux:multivariate-Lagrange}. Our contribution is to give a shorter proof, which takes advantage of the symmetries in the counting formulas for $d$-cacti.

\fig{width=.8\linewidth}{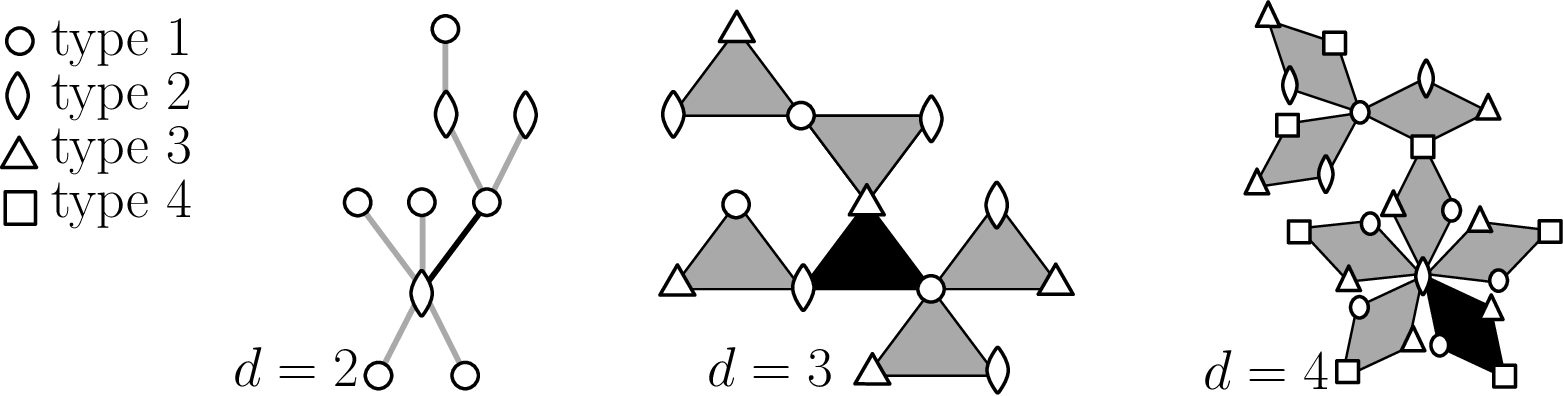}{Examples of rooted planar $d$-cacti for $d=2,3,4$. In all figures, the root $d$-gon is colored in black.}

We fix an integer $d\geq 2$. A $d$\emph{-graph} is a connected simple graph with $d$ types of vertices $1,2\ldots, d$, such that every 2-connected component is a cycle of length $d$ with vertices of type $1,2\ldots, d$ appearing in this order cyclically. Each such cycle is called a \emph{$d$-gon}. A \emph{plane embedding} of a $d$-graph is a drawing of $d$-graph in the plane without edge-crossings in such a way that every edge is incident to the unbounded region, and around each $d$-gon the vertices of type $1,2,..,d$ appear in clockwise order. A (planar) $d$-cacti is a plane-embedding of a $d$-graph considered up to continuous deformation (this is equivalent to fixing for each vertex $v$ the cyclic order of its incident $d$-gons). The \emph{size} of a $d$-cactus is its number of $d$-gons. Note that a $d$-cactus of size $n$ has $dn$ edges and, by the Euler relation, $(d-1)n+1$ vertices. The {\em degree} of a vertex is the number of incident $d$-gons. 

A $d$-cactus is \emph{rooted} if one of the $d$-gons is distinguished as the \emph{root $d$-gon}. Let $\nn=(n_1\ldots,n_d)$ be a tuple of positive integers. A $d$-cactus with $n_t$ vertices of type $t$ for all $t\in[d]$ is said to be {\em vertex-labeled} if the vertices of type $t$ have distinct labels in $[n_t]$ for all $t\in[d]$. The labeled vertices are denoted by $(t,i)$ where $t\in [d]$ is the type and $i \in [n_t]$ is the label. We denote by $\mC(\nn)$ the set of rooted vertex-labeled planar $d$-cacti with $n_t$ vertices of type $t$.
As mentioned earlier, the set $\mC(\nn)$ is empty unless $n:=\left(\sum_{t\in[d]} n_t -1 \right)/(d-1)$ is an integer which is the size of the cacti. Moreover we must impose $n_t\leq n$ for all $t\in[d]$ and we now suppose that these conditions hold.
For a tuple of integers $\NN=(N_{t,j})_{t\in[d],j>0}$ we denote by $\mC^\NN(\nn)$ the set of $d$-cacti in $\mC(\nn)$ having $N_{t,j}$ vertices of type $t$ and degree $j$ for all $t\in[d],j>0$. Of course this set is empty unless 
$$\sum_{j>0} N_{t,j}=n_t,~~ \textrm{ and } ~~\sum_{j>0} j\,N_{t,j}=n$$
for all $t\in[d]$. In~\cite{GJ92} Goulden and Jackson proved that under these conditions 
\begin{equation}\label{eq:GJ92}
|\mC^\NN(\nn)|=n^{d-1}\, \frac{\prod_{t\in[d]} n_t!(n_t-1)!}{\prod_{t\in[d],j>0}N_{t,j}!}.
\end{equation}
Actually the result in~\cite{GJ92} is stated in terms of rooted $d$-cacti with \emph{unlabeled} vertices, accounting for a factor $\prod_{t\in[d]} n_t!$ between~\eqref{eq:GJ92} and~\cite[Theorem 3.2]{GJ92}. 

We will now give a proof of~\eqref{eq:GJ92} using a philosophy similar to the one developed in previous sections. For a cactus $C\in\mC(\nn)$, we denote by $(t,i)$ the vertex of type $t$ labeled $i$.
For a tuple $\ga=(\ga_{t,i})_{t\in[d],i\in[n_t]}$ of integers we say that $C$ has \emph{degree vector} $\ga$ if the vertex $(t,i)$ has degree $\ga_{t,i}$ for all $t\in[d],i\in[n_t]$. We denote by $\mC_\ga$ the set of $d$-cacti in $\mC(\nn)$ having degree vector $\ga$. Of course this set is empty unless $\displaystyle \sum_{i\in[n_t]}\ga_{t,i}=n$ for all $t\in[d]$ and $\ga_{t,i}>0$ for all $t\in[d],i>0$. An equivalent way of stating~\eqref{eq:GJ92} is as follows.

\begin{theorem}\label{thm:GJ}
Let $n_1\ldots,n_d$ be positive integers such that $n:=\left(\sum_{t\in[d]} n_t -1 \right)/(d-1)$ is an integer, and $n_t\leq n$ for all $t\in[d]$.
Let $\ga=(\ga_{t,i})_{t\in[d],i\in[n_t]}$ be a tuple of positive integers such that $ \sum_{i\in[n_t]}\ga_{t,i}=n$ for all $t\in[d]$. Then the number of rooted vertex-labeled $d$-cacti with degree vector $\ga$ is
\begin{equation}\label{GJcacti}
|\mC_{\ga}| = n^{d-1}\, \prod_{t\in[d]} (n_t-1)!.
\end{equation}
\end{theorem}

\begin{corollary}\label{cor:totalnb}
The number of rooted vertex-labeled $d$-cacti of size $n$ with $n_t$ vertices of type $t$ for all $t\in[d]$ is
$$|\mC(\nn)|=n^{d-1}\, \prod_{t\in[d]} \binom{n-1}{n_t-1}(n_t-1)!,$$
if $\left(\sum_{t\in[d]} n_t -1 \right)/(d-1)=n$, and 0 otherwise.
\end{corollary}

\begin{proof}[Proof of Corollary~\ref{cor:totalnb}]
The right-hand side of~\eqref{GJcacti} does not depend on $\ga$. 
And since $|\mC(\nn)| = \sum_{\ga} |\mC_{\ga}|$, then it suffices to observe that for all $t\in[d]$ there are $\binom{n-1}{n_t-1}$ tuples of positive integers such that $\sum_{i\in[n_t]}\ga_{t,i}=n$.
\end{proof}

The rest of this section is devoted to the proof of Theorem~\ref{thm:GJ}. 
We start with an analogue of Lemma~\ref{lem:symmetry-Cayley-multitype}.

\begin{lemma} \label{lem:sym1}
Let $\nn=(n_1,\ldots,n_d)$ be a tuple of positive integers such that $n:=\left(\sum_{t\in[d]} n_t -1 \right)/(d-1)$ is an integer. 
Let $\ga=(\ga_{t,i})_{t\in[d],i\in[n_t]}$ and $\ga'=(\ga'_{t,i})_{t\in[d],i\in[n_t]}$ be tuples of positive of integers such that $ \sum_{i\in[n_t]}\ga_{t,i}=\sum_{i\in[n_t]}\ga'_{t,i}=n$ for all $t\in[d]$. 
Then the number of $d$-cacti in $\mC(\nn)$ having degree vector $\ga$ or $\ga'$ is the same: $\displaystyle |\mC_{\ga}| = |\mC_{\ga'}|$.
\end{lemma}

\fig{width=.8\linewidth}{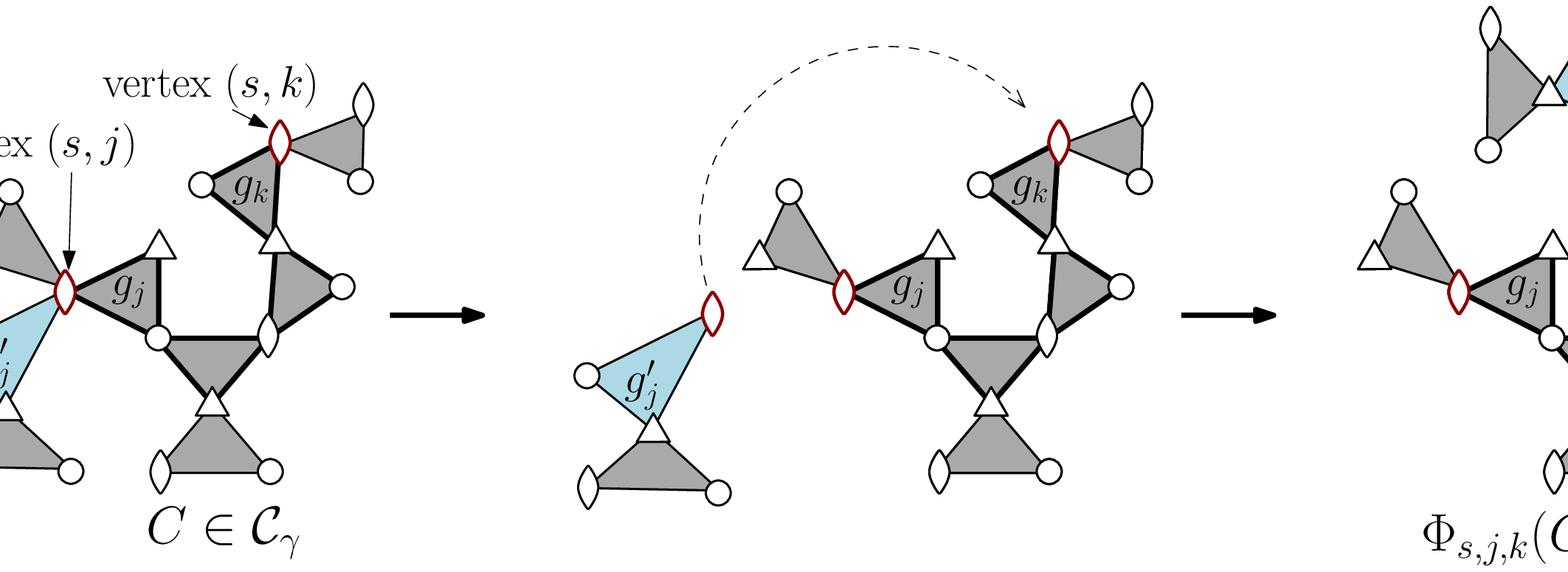}{The bijection $\Phi_{s,j,k}$. The chain of $d$-gons between vertices $(s,j)$ and $(s,k)$ is drawn in thick lines.} 

\begin{proof} It suffices to prove~$|\mC_{\ga}| = |\mC_{\ga'}|$ when the tuples $\ga$ and $\ga'$ only differ on two coordinates. Let 
$s\in[d]$, $j,k\in[n_s]$ such that $\ga_{s,j}>1$, let $\ga'$ be defined by $\ga'_{s,j}=\ga_{s,j}-1$, $\ga'_{s,k}=\ga_{s,k}+1$ and $\ga'_{t,i}=\ga_{t,i}$ for all $(t,i)\notin\{(s,j), (s,k)\}$. We need to prove $|\mC_{\ga}| = |\mC_{\ga'}|$ and we proceed to exhibit a bijection $\Phi_{s,j,k}$ between $\mC_{\ga}$ and $\mC_{\ga'}$. The bijection $\Phi_{s,j,k}$ is illustrated in Figure~\ref{fig:sym1-cacti} (the root $d$-gon plays no role in this construction and is not indicated). 

Let $C$ be a cactus in $\mC_{\ga}$. Let $P$ be the \emph{chain of $d$-gons} between the vertices $(s,j)$ and $(s,k)$, that is, the set of $d$-gons containing the edges of any simple path between the vertices $(s,j)$ and $(s,k)$ (this set of $d$-gons is independent of the simple path considered); see Figure~\ref{fig:sym1-cacti}. 
Let $g_j$ (resp. $g_k$) be the $d$-gon in $P$ incident to the vertex $(s,j)$ (resp. $(s,k)$). Let $g_j'$ be the $d$-gon incident to the vertex $(s,j)$ following $g_j$ clockwise around $(s,j)$. Observe that $g_j'\neq g_j$ since the vertex $(s,j)$ has degree $\ga_{s,j}>1$.
We define $\Phi_{s,j,k}(C)$ as the cactus obtained from $C$ by ungluing the $d$-gon $g_j'$ from the vertex $(s,j)$ and gluing it to $(s,k)$ in the corner following the $d$-gon $g_j$ clockwise around $(s,k)$; see Figure~\ref{fig:sym1-cacti}. It is clear that $\Phi_{s,j,k}(C)$ is a cactus in $\mC_{\ga'}$. It is equally clear that $\Phi_{s,j,k}$ is a bijection between $\mC_{\ga}$ and $\mC_{\ga'}$, since the inverse mapping is $\Phi_{s,j,k}^{-1}=\Phi_{s,k,j}$. This completes the proof.
\end{proof}

Given Lemma~\ref{lem:sym1} it is sufficient to prove~\eqref{GJcacti} for the particular tuple $\ga^*(\nn)=(\ga_{t,i})_{t\in[d],i\in[n_t]}$ defined by $\ga_{t,1}=n-n_t+1$ and $\ga_{t,i}=1$ for all $t$ in $[d]$ and all $i>1$ in $[n_t]$. Cacti in $\mC_{\ga^*(\nn)}$ are shown in Figure~\ref{fig:star-cactus}.
We will now invoke a second symmetry in order to enumerate~$\mC_{\ga^*(\nn)}$.

\begin{lemma} \label{lem:sym2}
Let $\nn=(n_1,\ldots,n_d)$ be a tuple of positive integers. Let $r,s\in[d]$ be such that $n_r>1$, and $n_s<n$ and let $\nn'=(n'_1,\ldots,n'_d)$ be defined by 
$n_r'=n_r-1$, $n_s'=n_s+1$ and $n_t'=n_t$ for all $t\notin\{r,s\}$. Then 
$$|\mC_{\ga^*(\nn)}|/(n_r-1)=|\mC_{\ga^*(\nn')}|/(n_s'-1).$$
\end{lemma}

\fig{width=.8\linewidth}{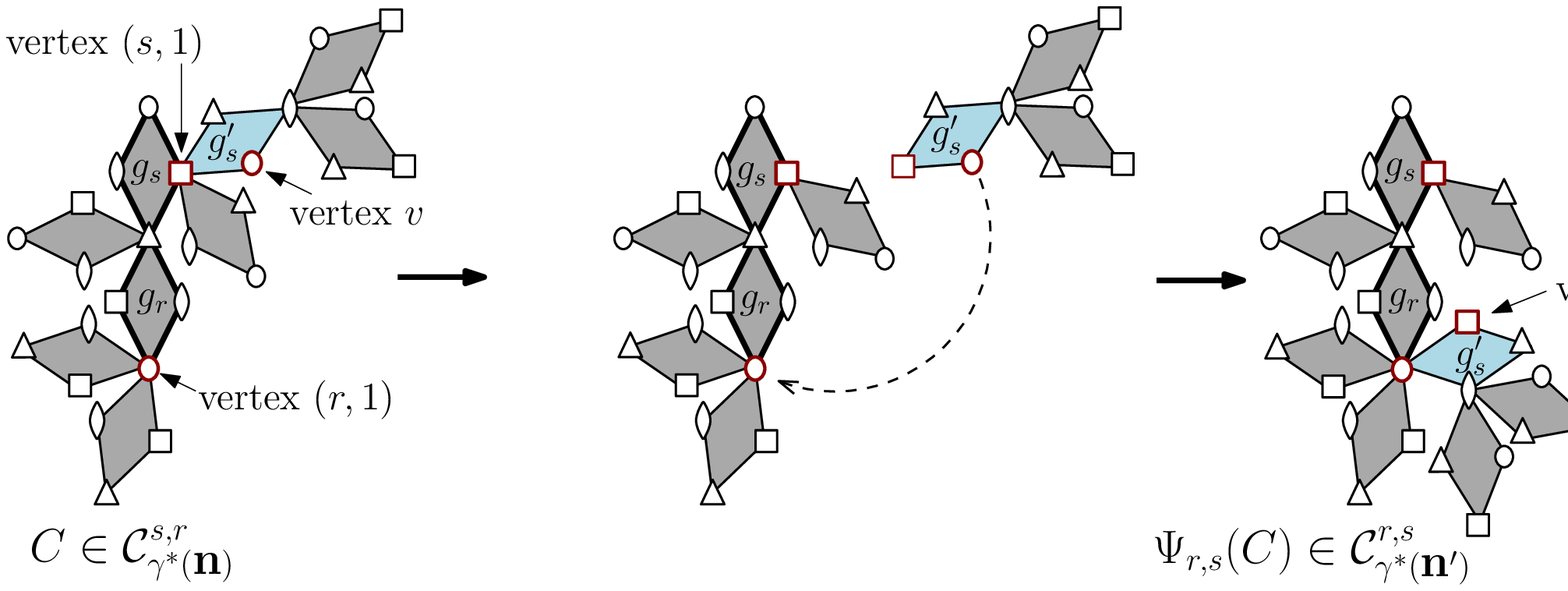}{The bijection $\Psi_{r,s}$. The chain of $d$-gons between vertices $(r,1)$ and $(s,1)$ is drawn in thick lines. The $d$-gon $g_s'$ is indicated in lighter shade. The vertex $v=(r,n_r)$ and $v'=(s,n_s+1)$.}

\begin{proof}
Let $C$ be a cactus in $\mC_{\ga^*(\nn)}$. Let $P$ be the \emph{chain of $d$-gons} between vertices $(r,1)$ and $(s,1)$ (the set of $d$-gons containing the edges of any simple path between these vertices). Let $g_s$ be the unique $d$-gon in $P$ incident to the vertex $(s,1)$ and let $g_s'$ be the $d$-gon incident to $(s,1)$ following $g_s$ clockwise around $(s,1)$. Observe that $g_s'\neq g_s$ since the vertex $(s,1)$ has degree $n-n_s+1>1$. Hence, the vertex $v$ of type $r$ incident to $g_s'$ has a label distinct from $1$, and has degree 1. We say that the cactus $C\in \mC_{\ga^*(\nn)}$ is in ${\mC}^{s,r}_{\ga^*(\nn)}$ if $v$ has label $n_r$. It is clear that $\displaystyle |{\mC}^{s,r}_{\ga^*(\nn)}|=|\mC_{\ga^*(\nn)}|/(n_r-1)$. Hence to prove the lemma it suffices to exhibit a bijection $\Psi_{r,s}$ between ${\mC}^{s,r}_{\ga^*(\nn)}$ and ${\mC}^{r,s}_{\ga^*(\nn')}$. This bijection is represented in Figure~\ref{fig:sym2-cacti}. 

Let $C$ be a cactus in ${\mC}^{s,r}_{\ga^*(\nn)}$. Let $P$, $g_s,g_s'$ and $v$ be as above. We also denote by $g_r$ be the unique $d$-gon in the chain $P$ incident to the vertex $(r,1)$. We then denote by $\Psi_{r,s}(C)$ the cactus obtained from $C$ by ungluing the $d$-gon $g_s'$ from the vertex $(s,1)$, and gluing it to the vertex $(r,1)$ in the corner following the $d$-gon $g_r$ around $(r,1)$; see Figure~\ref{fig:sym2-cacti}. In this process, the vertex $v$ becomes identified with the vertex $(s,1)$ (so that the label $n_r$ of $v$ disappears), while the vertex $v'$ of type $s$ incident to the $d$-gon $g_s'$ (previously identified with $(s,1)$) takes the label $n_s+1$. It is clear that $\Psi_{r,s}(C)$ is a cactus in ${\mC}^{r,s}_{\ga^*(\nn')}$. It is equally clear that $\Psi_{r,s}$ is a bijection between ${\mC}^{s,r}_{\ga^*(\nn)}$ and ${\mC}^{r,s}_{\ga^*(\nn')}$, since the inverse mapping is $\Psi_{r,s}^{-1}=\Psi_{s,r}$. This completes the proof.
\end{proof}

\fig{width=.7\linewidth}{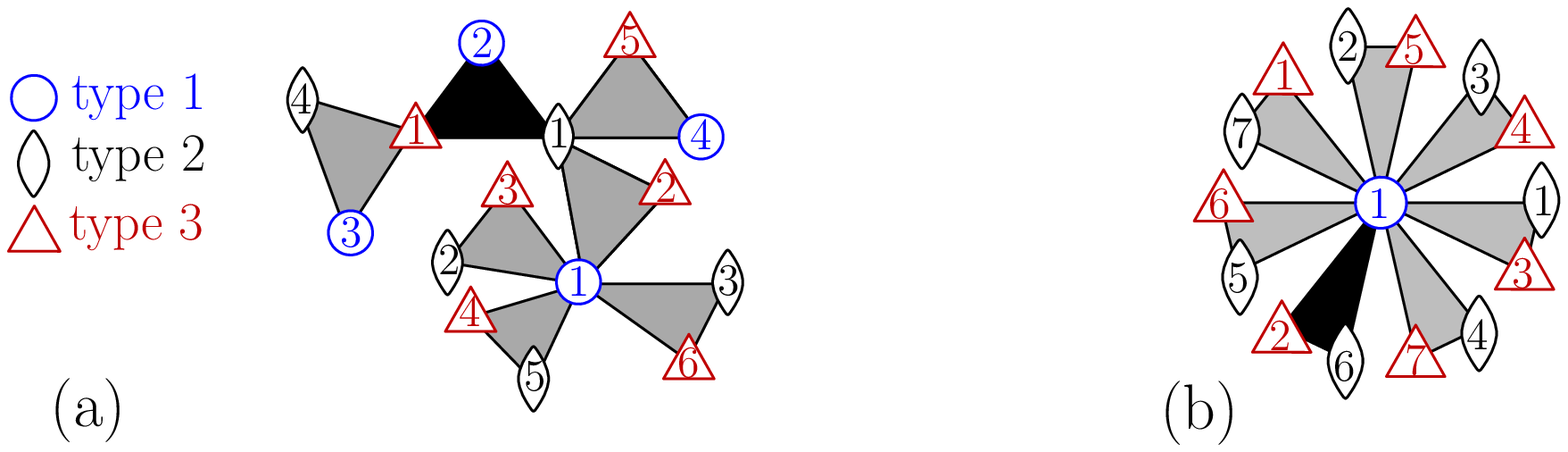}{(a) A cactus in $\mC_{\ga^*(\nn)}$, for $\nn=(4,4,5)$. (b) A cactus in $\mC_{\ga^*(\nn)}$, for $\nn=(1,7,7)$. In both figures the root $d$-gon is colored in black and the numbers indicate the labels of vertices.}

By Lemma~\ref{lem:sym2}, we immediately get that for any tuple $\nn=(n_1,\ldots,n_d)$ in $[n]^d$ such that $n:=\left(\sum_{t\in[d]} n_t -1 \right)/(d-1)$ is an integer,
$$|\mC_{\ga^*(\nn)}|=\frac{\prod_{t\in[d]} (n_t-1)!}{(n-1)!^{d-1}}\,|\mC^*_{\ga^*(1,n,n,\ldots,n)}|.$$
Moreover the set of cacti $\mC^*_{\ga^*(1,n,n,\ldots,n)}$ is very easy to enumerate. Indeed this is the set of cacti where the unique vertex of type 1 has degree $n$ and all the other vertices have degree 1; see Figure~\ref{fig:star-cactus}(b). The only freedom left in such a cactus are the labels of the $n$ vertices of type $t$ for all $t\in\{2,\ldots,d\}$. This gives 
$$|\mC^*_{\ga^*(1,n,n,\ldots,n)}|=n!^{d-1},$$
hence $|\mC_{\ga^*(\nn)}|=n^{d-1}\prod_{t\in[d]} (n_t-1)!.$
Together with Lemma~\ref{lem:sym1} this completes the proof of Theorem~\ref{thm:GJ}.\\

\noindent\textbf{Acknowledgment:} We thank Mireille Bousquet-Mélou for pointing out an error in a previous version of this manuscript, Ira Gessel for mentioning reference~\cite{Knuth:enumeration-trees}, and to anonymous referees for valuable suggestions.

\bibliographystyle{plain} 
\bibliography{biblio-trees} 
\label{sec:biblio} 
 

\end{document}